\newif\ifdraft
\newif\ifarxiv
\arxivtrue

\documentclass[reqno,12pt]{amsart}
\usepackage[a4paper, margin=2cm]{geometry}

\ifdraft

\usepackage[inline]{showlabels}

\fi

\usepackage{amsmath, amsthm, thmtools, amssymb, mathtools, amsaddr}
\usepackage{enumitem}

\usepackage[dvipsnames]{xcolor}
\usepackage{hyperref}
\hypersetup{
	colorlinks = true,
	linkcolor = {BlueViolet},
	urlcolor={RoyalBlue},
	citecolor={BrickRed}
}

\makeatletter
\renewcommand{\itemautorefname}{\@gobble}
\makeatother

\usepackage{float}
\usepackage{import}
\usepackage{xifthen}
\usepackage{pdfpages}
\usepackage{transparent}

\theoremstyle{plain}
\newtheorem{theorem}{Theorem}[section]
\newtheorem*{theorem*}{Theorem}
\newtheorem{theoremX}{Theorem}

\newtheorem{prop}[theorem]{Proposition}
\newtheorem{corollary}[theorem]{Corollary}
\newtheorem{lemma}[theorem]{Lemma}

\theoremstyle{definition}
\newtheorem{defn}[theorem]{Definition}
\newtheorem{example}[theorem]{Example}

\newtheorem*{conj*}{Conjecture}
\newtheorem{remark}[theorem]{Remark}

\DeclareMathOperator{\codim}{codim}

\begin{document}
\allowdisplaybreaks
\title[Distance to the Cut Locus]{Distance from a Finsler Submanifold to its Cut Locus and the Existence of a Tubular Neighborhood}

\author[A. Bhowmick]{Aritra Bhowmick}
\address{Department of Mathematics, Indian Institute of Science, Bengaluru, India}
\email{aritrab@iisc.ac.in}
\author[S. Prasad]{Sachchidanand Prasad}
\address{School of Mathematics, Jilin University, China}
\address{Faculty of Mathematics and Computer Science, G\"{o}ttingen University, Germany}
\email{sachchidanand.prasad1729@gmail.com}



\begin{abstract}
    In this article we prove that for a closed, not necessarily compact, submanifold $N$ of a possibly non-complete Finsler manifold $(M, F)$, the cut time map is always positive. As a consequence, we prove the existence of a tubular neighborhood of such a submanifold. When $N$ is compact, it then follows that there exists an $\epsilon > 0$ such that the distance between $N$ and its cut locus $\mathrm{Cu}(N)$ is at least $\epsilon$. This was originally proved by B. Alves and M. A. Javaloyes (\emph{Proc. Amer. Math. Soc.} 2019). We have given an alternative, rather geometric proof of the same, which is novel even in the Riemannian setup. We also obtain easier proofs of some results from N. Innami et al. (\emph{Trans. Amer. Math. Soc.}, 2019), under weaker hypothesis.
\end{abstract}

\date{\today}
\maketitle


\setcounter{tocdepth}{3}

\frenchspacing 

\section{Introduction} \label{sec:introduction}
In the field of Riemannian geometry, \emph{geodesics} are curves that are locally distance minimizing. If a connected Riemannian manifold $(M, g)$ is assumed to be complete, then between any two points, say, $p, q \in M$, there always exists a geodesic which is globally distance minimizing, which we simply say is a \emph{minimizer}. This leads to the notion of the \emph{cut locus} of a point $p$, originally introduced by Henri Poincar\'{e} \cite{Poin05}. The cut locus $\mathrm{Cu}(p)$ of a point $p$ is the set consisting of all points $q$ such that there exists a minimizer joining $p$ to $q$, any extension of which fails to be globally distance minimizing. Cut locus has been extensively studied in the literature \cite{Kob67,Buc77,Wol79,Sak96}. In particular, it is well-known that $\mathrm{Cu}(p)$ is the closure of the collection of those points $q$ for which there are at least two distinct minimizers joining $p$ to $q$. Now, the classical Whitehead convexity theorem \cite{Whitehead32,ChEb75} states that around every point $p \in M$, there exists a sufficiently small metric ball, say, $B$ such that for any $q \in B$ there exists a unique minimizer joining $p$ to $q$. As an immediate consequence, one concludes that $p \not \in \mathrm{Cu}(p)$.

The notion of cut locus readily generalizes to submanifolds $N \subset M$. Given any point $q \in M$, an $N$-segment is a minimizer joining a point in $N$ to $q$, with length equal to the distance from $N$ to $q$. The cut locus of $N$, denoted $\mathrm{Cu}(N)$, consists of those points $q$ such that there is an $N$-segment joining $N$ to $q$, any extension of which fails to be distance minimizing from $N$. It requires some effort to prove that $\mathrm{Cu}(N)$ is disjoint from $N$. The standard proof involves looking at the \emph{normal exponential map} associated to $N$, which is the usual exponential map restricted to the normal bundle of $N$. One first proves that it is an immersion near the $0$-section, and then, one shows that it is an embedding in some sufficiently small neighborhood \cite{Lee18}. Such a neighborhood is known as a \emph{tubular neighborhood}. As a consequence, it follows that $N \cap \mathrm{Cu}(N) = \emptyset$. In \cite{Oneil83}, the author proves this fact for \emph{semi}-Riemannian manifolds. We also note that in the context of sub-Riemannian geometry, there are examples where $p \in \mathrm{Cu}(p)$ for a point $p$ in a sub-Riemannian manifold \cite{Mont02}. \medskip

Finsler geometry is a natural generalization of the Riemannian one, which was first studied by P. Finsler in his dissertation \cite{Fin51}. A \emph{Finsler metric} $F$ on a manifold $M$ is a parametrized collection of Minkowski norms on each tangent space, which, unlike a Riemannian norm, may not be induced by an inner product. In this generality one needs to overcome certain challenges, nevertheless, most of the standard notions from Riemannian geometry find their natural counterparts in the Finsler setting, see \cite{AbaPat94,Bao2000,Shen01,Ohta2021,BhoPra2023} for a survey of results. In particular, geodesics and the notion of cut locus are readily defined. Furthermore, Whitehead convexity theorem still holds \cite{Bao2000}, and one immediately obtains that the cut locus of a point in a Finsler manifold is disjoint from the point itself.

One of the primary difficulties in dealing with a submanifold $N$ of a Finsler manifold $(M, F)$ is that, in the absence of an inner product, the natural replacement of a normal bundle of $N$ is no longer a vector bundle; it is only a cone bundle, and it may have singularities at the zero section. Consequently, if one attempts to prove that $N \cap \mathrm{Cu}(N) = \emptyset$, the proof technique from the Riemannian geometry immediately breaks down, since the normal exponential map is not even differentiable near the zero section. \medskip

The primary goal of this article is to prove \autoref{thm:cutTimePositive}, and as an immediate corollary we obtain the following.

\begin{theoremX}[\autoref{cor:cutLocusDisjoint}]
    Given a closed (not necessarily compact) submanifold $N$ of a Finsler manifold $(M, F)$, we have $N \cap \mathrm{Cu}(N) = \emptyset$.
\end{theoremX}

In fact, when $N$ is compact, in \autoref{cor:injectivityRadiusPositive} we show that $d(N, \mathrm{Cu}(N)) > 0$. This follows immediately from the next theorem. Note that the notion of a tubular neighborhood of a submanifold can be naturally extended to the Finsler context (\autoref{defn:geometricTubularNbd}).

\begin{theoremX}[\autoref{thm:tubularNBD}]
    Given a closed submanifold $N$ of a Finsler manifold $(M, F)$, there exists a tubular neighborhood of $N$. If $N$ is compact, then we can get a tubular neighborhood whose image under the normal exponential map is $\left\{ x \in M \;\middle|\; d(N, x) < \epsilon \right\}$ for some $\epsilon > 0$, where $d(N,\_ )$ is the distance function from $N$.
\end{theoremX}

We would like to point out that most of the results in this article, including the ones mentioned above, were originally proved in \cite{Alves2019}, albeit with a different approach. See \autoref{sec:javaloyes} for a further discussion.

\subsection*{Conventions} Given a bundle $E$ over a manifold, $\Gamma E$ denotes the sheaf of sections of $E$, and in particular, $X \in \Gamma E$ denotes a \emph{local} section of $E$ over an unspecified open set in $M$. Boldface letters, e.g., $\mathbf{v},\mathbf{n}$ etc. will always denote tangent vectors. Unless stated explicitly, we have not assumed any completeness assumption on the Finsler metric. As such, whenever we have considered a geodesic, it is tacitly assumed to be defined in the maximal possible domain. All manifold are connected, and without boundary. By a \emph{closed} submanifold, we shall mean topologically closed, but not necessarily compact, embedded submanifold.

\subsection*{Organization of the paper} In \autoref{sec:preliminaries}, we recall some basic definitions and results from Finsler geometry\ifarxiv, while deferring the proof of a technical result (\autoref{lemma:smallBallPrincipalCurvature}) to the \autoref{sec:curvatureSmallBall}\fi. Then, in \autoref{sec:tubularNBD} we obtain the main results of this article. Finally, in \autoref{sec:existingLit} we discuss how this article relates to \cite{Alves2019} and \cite{InnItoNagShi19}.
\section{Preliminaries on Finsler Geometry} \label{sec:preliminaries}
This section recalls some necessary concepts and definitions in Finsler geometry and cut locus, while deferring to \cite{Shen01, Javaloyes2015, Ohta2021,BhoPra2023} for details.
\subsection{Finsler Metric} Let us recall the definition.
\begin{defn}\label{defn:FinslerMetric}
    Let $M$ be a smooth manifold, and $TM$ denotes its tangent bundle. A \textit{Finsler metric} on $M$ is a continuous function $F: TM \to \mathbb{R}$ having the following properties.
    \begin{enumerate}
        \item $F$ is smooth on $\widehat{TM} \coloneqq TM \setminus \mathbf{0}$.
        \item For any $p\in TM$, the restriction $F_p\coloneqq F\big|_{T_pM}$ is a Minkowski norm, i.e.,
        \begin{enumerate}
            \item for any $\lambda>0$ and $\mathbf{v}\in T_pM\setminus\{0\}$, we have $F_p(\lambda \mathbf{v})=\lambda F_p(\mathbf{v})$, and 
            \item for all $\mathbf{v}\in T_pM\setminus\{0\}$, the symmetric tensor $g_{\mathbf{v}}$ on $T_pM$, called the \emph{fundamental tensor}, is positive definite, where 
            \begin{displaymath}
                g_{\mathbf{v}}(\mathbf{v}_1,\mathbf{v}_2)\coloneqq \left.\dfrac{1}{2} \dfrac{\partial^2}{\partial s_1 \partial s_2} \right|_{(s_1, s_2) = (0,0)} \left(F_p(\mathbf{v} + s_1\mathbf{v}_1 + s_2\mathbf{v}_2)\right)^2.
            \end{displaymath}
        \end{enumerate}
    \end{enumerate}
    $F$ is \emph{reversible} if $F(-\mathbf{v}) = F(\mathbf{v})$ holds for all $\mathbf{v} \in \widehat{TM}$.
\end{defn}

We have the useful identity involving the fundamental tensor
\begin{equation}\label{eq:fundTensorIdentity}
    g_\mathbf{v}(\mathbf{v},\mathbf{v}) = F_p(\mathbf{v})^2, \quad \mathbf{v} \in T_p M \setminus 0.
\end{equation}
Given $\mathbf{v} \in T_p M \setminus 0$, the associated \emph{Cartan tensor} on $T_p M$ is a symmetric $3$-tensor defined as
\[C_{\mathbf{v}}(\mathbf{v}_1,\mathbf{v}_2,\mathbf{v}_3) \coloneqq \left. \frac{1}{4} \frac{\partial^3}{\partial s_1 \partial s_2 \partial s_3} \right|_{s_1=s_2=s_3=0} \left( F_p(\mathbf{v} + s_1 \mathbf{v}_1 + s_2 \mathbf{v}_2 + s_3 \mathbf{v}_3) \right)^2.\]
For each $\mathbf{v}\in T_p M \setminus 0$ and $\mathbf{u}, \mathbf{w}\in T_p M$, we have
\begin{equation}\label{eq:cartanTensorIdentity}
    C_{\mathbf{v}}(\mathbf{v}, \mathbf{u},\mathbf{w}) = C_{\mathbf{v}}(\mathbf{u}, \mathbf{v}, \mathbf{w}) = C_{\mathbf{v}}(\mathbf{u}, \mathbf{w}, \mathbf{v}) = 0.
\end{equation}
One extends the definition of fundamental tensor and the Cartan tensor to include vector fields, so that $g_V(X, Y)$ and $C_V(X, Y, Z)$ makes sense for any $V\in \Gamma \widehat{TM}$ and $X,Y, Z \in \Gamma TM$.

\subsubsection{Legendre Transformation} We have a \emph{nonlinear} fiber preserving map $\mathcal{L} : \widehat{TM} \rightarrow  \widehat{T^*M}$ given by
\begin{equation}\label{eq:legendreTransformation}
    \mathcal{L}(\mathbf{v})(\mathbf{w}) \coloneqq g_{\mathbf{v}}(\mathbf{v}, \mathbf{w}), \quad \mathbf{v}\in T_p M \setminus 0, \; \mathbf{w} \in T_p M.
\end{equation}
$\mathcal{L}$ is called the \emph{Legendre transformation} associated to $(M,F)$. It follows that $\mathcal{L}$ is a $C^\infty$-diffeomorphism. We extend $\mathcal{L}$ to all of $TM$ by setting $\mathcal{L}(0) = 0$. The extension $\mathcal{L}: TM \rightarrow T^*M$ is then a fiber-preserving \emph{homeo}morphism.

\subsubsection{Chern Connection} Unlike the Levi-Civita connection in the Riemannian context, there are several canonically defined connections on Finsler manifold. In this article, we consider the Chern connection.

\begin{defn}\label{defn:chernConnection}\cite{Rademacher2004, Javaloyes2014}
    For each $V \in \Gamma \widehat{TM}$, we have a unique affine connection \[\nabla^V: \Gamma TM \otimes \Gamma TM \rightarrow \Gamma TM,\] called the \emph{Chern connection}, satisfying the following conditions for any $X, Y, Z \in \Gamma TM$.
    \begin{itemize}
        \item $\nabla^V_X Y - \nabla^V_Y X = [X, Y]$.
        \item $X (g_V(Y,Z)) = g_V(\nabla^V_X Y, Z) + g_V(Y, \nabla^V_X Z) + 2 C_V(\nabla^V_X V, Y, Z)$.
    \end{itemize}
\end{defn}

The value for $\nabla^V_X Y|_p$ depends only on the values of $V(p), X(p)$ and the values of $Y$ on a curve with initial velocity $X(p)$. The associated \emph{curvature tensor} is given by
\begin{equation}\label{eq:curvatureTensor}
    R^V(X,Y) Z \coloneqq \nabla^V_X \nabla^V_Y Z - \nabla^V_Y \nabla^V_X Z - \nabla^V_{[X,Y]} Z, \quad X,Y,Z \in \Gamma TM,
\end{equation}
which is skew-symmetric: $R^V(X,Y) = -R^V(Y,X)$. For the closely related concept of Chern curvature, we refer to \cite{Javaloyes2014a,Javaloyes2020}. 

\begin{defn}\label{defn:flagCurvature} \cite{Rademacher2004}
    Given $V \in \Gamma \widehat{TM}$, the \emph{flag curvature} of a $2$-plane field $\sigma \in \Gamma \textrm{Gr}_2 TM$ containing $V$, is given by
\begin{equation}\label{eq:flagCurvature}
    K^V(\sigma) \coloneqq \frac{g_V(R^V(V,W)W, V)}{g_V(V, V) g_V(W, W) - g_V(V, W)^2}, 
\end{equation}
where $W \in \Gamma \widehat{TM}$ is such that $\sigma = \mathrm{Span}\left\langle V, W \right\rangle$. 
\end{defn}

\subsubsection{Geodesics}
Let us recall the length and energy functionals defined on $\mathcal{P} = \mathcal{P}([a,b])$, the space of piecewise smooth paths $\gamma : [a,b] \rightarrow M$, given by
\[L(\gamma) \coloneqq \int_a^b F(\dot \gamma(t)) dt, \qquad E(\gamma)\coloneqq \frac{1}{2}\int_a^b F(\dot \gamma(t))^2 dt.\]
A piecewise $C^1$ curve $\gamma : [a,b] \rightarrow  M$ is called a \emph{geodesic} if it is a critical point of the energy functional in the variational sense. Geodesics are the solution to a second order ordinary differential equation, called the \emph{geodesic equation} \cite[Eq. 3.15]{Ohta2021}, and hence, they are always smooth. Furthermore, given a vector $\mathbf{v} \in T_p M$, we have a unique maximal geodesic 
\begin{equation}\label{eq:uniqueMaximalGeodesic}
    \gamma_{\mathbf{v}} : [0,\ell] \rightarrow  M
\end{equation}
satisfying the initial value conditions $\gamma(0) = p$ and $\dot \gamma(0) = \mathbf{v}$.

A Finsler manifold $(M, F)$ is said to be \emph{forward complete} if $\gamma_{\mathbf{v}}$ is defined for all time $[0, \infty)$. We say $(M, F)$ is \emph{backward complete} if the reverse Finsler metric $\bar{F}$ is forward complete. The \emph{exponential map} is defined as
\begin{equation}\label{eq:exponentialMap}
    \begin{aligned}
        \exp : \mathcal{O} \subset TM &\rightarrow M \\
        \mathbf{v} &\mapsto \gamma_{\mathbf{v}}(1),
    \end{aligned}
\end{equation}
where $\mathcal{O}$ is the maximal open set of $TM$ containing the $0$-section, for which $\gamma_{\mathbf{v}}(1)$ is defined. Clearly, $F$ is complete precisely when $\mathcal{O} = TM$. It follows from the theory of ordinary differential equation that the exponential map is smooth on $\mathcal{O} \cap \widehat{TM}$, but only $C^1$ on $\mathcal{O}$; see \cite[Section 5.3]{Bao2000} for details. Unless explicitly stated, we have not assumed any completeness throughout this article. Whenever we consider a geodesic, it is understood that we only consider it in the domain where it is defined.

A geodesic is locally distance minimizing, where the \emph{Finsler distance} between points $p, q \in M$ is defined as
\begin{equation}\label{eq:finslerDistance}
    d(p, q) = \inf \left\{ L(\gamma) \;\middle|\; \gamma \in \mathcal{P}, \gamma(a) = p, \gamma(b) = q \right\}.
\end{equation}
Unless $F$ is reversible, the distance function is asymmetric (i.e., it is a \emph{pseudo}metric), although it still induces the same topology as that of the underlying manifold.

\begin{defn}\label{defn:minimizer}
    A geodesic $\gamma : [0,\ell] \rightarrow  M$ is said to be a \emph{global distance minimizer} (or simply a \emph{minimizer}) if $\gamma$ is unit-speed and $d(\gamma(0), \gamma(\ell)) = \ell = L(\gamma)$.
\end{defn}

Given a geodesic $\gamma : [0, \ell] \rightarrow M$, the reversed curve $\bar{\gamma} : [0, \ell] \rightarrow M$ given by $\bar{\gamma}(t) = \gamma(\ell - t)$ need not be a geodesic. Nevertheless, $\bar{\gamma}$ is a geodesic with respect to the reverse Finsler metric $\bar{F}$ \cite[Section 2.5]{Ohta2021}. It follows from the Hopf-Rinow theorem \cite[Theorem 3.21]{Ohta2021}, that whenever $F$ is forward (or backward) complete, given any two points $p, q \in M$ there exists a minimizer with respect to $F$, and a possibly distinct minimizer with respect to $\bar{F}$, joining $p$ to $q$.

\subsection{Submanifolds in Finsler Manifolds}
Given a submanifold $N$ of a Finsler manifold $(M, F)$, the normal cone bundle is considered as the natural replacement for normal bundles.

\begin{defn}\label{defn:normalCone}
    Given a submanifold $N \subset M$, the set $$\nu_p = \nu_p(N) = \big\{ \mathbf{v} \in T_p M \setminus \left\{ \mathbf{0} \right\} \;\big|\; g_{\mathbf{v}}(\mathbf{v},\mathbf{w}) = 0 \; \forall \mathbf{w} \in T_p N\big\} \cup \left\{ \mathbf{0} \right\}$$
    is called the \emph{normal cone} of $N$ at $p \in N$. The set $\nu = \nu(N) = \cup_{p \in N} \nu_p(N)$ is called the \emph{normal cone bundle} of $N$. The \emph{unit normal cone bundle} of $N$ is denoted as $S(\nu) = \cup_{p \in N} S(\nu_p)$, where $S(\nu_p) \coloneqq \left\{ \mathbf{v} \in \nu_p \;\middle|\; F_p(\mathbf{v}) = 1 \right\}$.
\end{defn}

The normal bundle $\nu(N)$ is \emph{not} a vector bundle in general. For any $\mathbf{n} \in \nu_p$, we have $\lambda \mathbf{n} \in \nu_p$ for $\lambda \ge 0$, i.e., each fiber $\nu_p$ is cone-like. Thus, in general, $\nu$ is not smooth as we can possibly have singularity at the $0$-section $\mathbf{0}$. The \emph{slit} normal cone bundle is defined as $\hat{\nu} \coloneqq \nu \setminus \mathbf{0}$. The Legendre transformation $\mathcal{L} : TM \rightarrow T^*M$ maps $\nu$ homeomorphically onto the annihilator bundle of $TN$, which is a genuine smooth vector bundle. As a consequence, $\nu$ is only a closed \emph{topological} submanifold of $TM$, with dimension $\dim N + \codim N = \dim M$. On the other hand, $\hat{\nu}$ (resp. $S(\nu)$) is a \emph{smooth} submanifold of $TM$ of dimension $\dim M$ (resp. $\dim M - 1$).  Note that for any $\mathbf{n} \in \hat{\nu}$, we can define a smooth local extension $\tilde{\mathbf{n}} \in \Gamma \hat{\nu}$ of $\mathbf{n}$ via $\mathcal{L}$.

\subsubsection{Principal Curvature}
Given a vector $0 \ne \mathbf{n} \in T_p M$ at some $p \in N$, we have a canonical decomposition 
\[T_p M = T_p N \oplus \left( T_p N  \right)^{\perp_{g_{\mathbf{n}}}},\]
where $\left( T_p N \right)^{\perp_{g_{\mathbf{n}}}} = \left\{ \mathbf{v} \in T_p M \;\middle|\; g_{\mathbf{n}}(\mathbf{v},\mathbf{w}) = 0 \; \forall \mathbf{w}\in T_p N \right\}$. Any $\mathbf{v} \in T_p M$ then has a canonical decomposition 
\begin{equation}\label{eq:canonicalDecomposition}
    \mathbf{v} = \mathbf{v}^{\top_{\mathbf{n}}} + \mathbf{v}^{\perp_{\mathbf{n}}}.
\end{equation}
Note that $\mathbf{n} \in \left( T_p N \right)^{\perp_{g_{\mathbf{n}}}}$ if and only if $\mathbf{n} \in \hat{\nu}_p \coloneqq  \nu_p \setminus \left\{ 0 \right\}$. 
\begin{defn}\label{defn:secondFundamentalForm}
    Given $\mathbf{n} \in \hat{\nu}_p$, the \emph{second fundamental form} of $N$ at $p$ along $\mathbf{n}$ is defined as 
    \begin{equation}\label{eq:secondFundForm}
        \begin{aligned}
            \Pi^{\mathbf{n}} : T_p N \otimes T_p N &\rightarrow \left( T_p N \right)^{\perp_{g_{\mathbf{n}}}} \\
            \mathbf{x} \otimes \mathbf{y} &\mapsto \left( \nabla^{\mathbf{n}}_{X} Y \middle|_p \right)^{\perp_{\mathbf{n}}},
        \end{aligned}
    \end{equation}
    where $X, Y \in \Gamma TM$ are some arbitrary (local) extensions of $\mathbf{x}, \mathbf{y}$, which are tangent to $N$ at points of $N$. 
\end{defn}

\begin{prop}\label{prop:secondFundFormTensor}
    The second fundamental form is a symmetric $2$-tensor, independent of any choice of extensions.
\end{prop}
\ifarxiv \ifdraft {\color{blue} \fi
\begin{proof}
    Let us first show that the expression $\left( \nabla^{\mathbf{n}}_X Y \middle|_p \right)^{\perp_{\mathbf{n}}}$ is symmetric in $X, Y$. Indeed, since $\nabla$ is torsion free, we have
    \[\left( \nabla^{\mathbf{n}}_X Y \middle|_p \right)^{\perp_{\mathbf{n}}} - \left( \nabla^{\mathbf{n}}_Y X \middle|_p \right)^{\perp_{\mathbf{n}}} = \left( [X, Y]_p \right)^{\perp_{\mathbf{n}}}.\]
    Since $X, Y$ are tangent to $N$ at points of $N$, so is their Lie bracket, and in particular, $[X, Y]_p \in T_p N$. But then, $\left( [X, Y]_p \right)^{\perp_{\mathbf{n}}} = 0$, proving the symmetry. Clearly, the value of $\nabla^{\mathbf{n}}_X Y$ at $p$ depends on $\mathbf{x} = X(p)$, and hence by the above symmetry, on $\mathbf{y} = Y(p)$. Lastly, it is easy to see that $\Pi^{\mathbf{n}}$ is $C^\infty(N)$-linear in the first variable, and by symmetry, in the second variable as well. Thus, $\Pi^{\mathbf{n}}$ is a symmetric $2$-tensor on $T_p N$, taking values in $\left( T_p N \right)^{\perp_{g_{\mathbf{n}}}}$.
\end{proof}
\ifdraft } \fi \fi

Since $g_{\mathbf{n}}$ is nondegenerate on $T_p N$, we can now define the following.

\begin{defn}\label{defn:shapeOperator}
    Given $0 \ne \mathbf{n} \in \nu_p(N)$, the \emph{shape operator} of $N$ at $p$ along $\mathbf{n}$ is a linear map $A_{\mathbf{n}} : T_p N \rightarrow T_p N$ determined by the equation
    \begin{equation}\label{eq:shapeOperator}
        g_{\mathbf{n}} \left( A_{\mathbf{n}} \mathbf{x}, \mathbf{y}\right) = g_{\mathbf{n}} \left( \mathbf{n}, \Pi^{\mathbf{n}}(\mathbf{x}, \mathbf{y}) \right), \qquad \mathbf{x},\mathbf{y} \in T_p N.
    \end{equation}
\end{defn}

\begin{prop}\label{prop:shapeOperator} \cite[Lemma 2.9]{BhoPra2024}
    For any $0 \ne \mathbf{n} \in \nu_p$ and $\mathbf{x} \in T_p N$, we have 
    \[A_{\mathbf{n}}(\mathbf{x}) = - \left( \nabla^{\mathbf{n}}_{\mathbf{x}} \tilde{\mathbf{n}} \middle|_p \right)^{\top_{\mathbf{n}}},\]
    where $\tilde{\mathbf{n}} \in \Gamma \hat{\nu}$ is an arbitrary local extension of $\mathbf{n}$.
\end{prop}

Note that in \cite{BhoPra2024}, the second fundamental form was defined with a negative sign. As $\Pi^\mathbf{n}$ is symmetric, it follows from \autoref{eq:shapeOperator} that $A_{\mathbf{n}}$ is self-adjoint with respect to $g_{\mathbf{n}}$. Consequently, $A_{\mathbf{n}}$ is orthogonally diagonalizable over the reals.

\begin{defn}\label{defn:principalCurvature}
    Given $0 \ne \mathbf{n} \in \nu_p(N)$, the eigenvalues of $A_{\mathbf{n}}$ are called the \emph{principal curvatures} of $N$ at $p$ along $\mathbf{n}$. The \emph{absolute principal curvature} of $N$ at $p$ along $\mathbf{n}$ is defined as $\max_{1 \le i \le \dim N} \left\lvert \kappa_i \right\rvert$, where $\kappa_i$ are the principal curvatures.
\end{defn}

Let $\kappa$ be some eigenvalue of $A_{\mathbf{n}}$, and $0 \ne \mathbf{v} \in T_p N$ be an eigenvector corresponding to $\kappa$. Then, 
\begin{equation}\label{eq:eigenValue}
    g_{\mathbf{n}}(A_{\mathbf{n}} \mathbf{v}, \mathbf{v}) = \kappa g_{\mathbf{n}}(\mathbf{v}, \mathbf{v}) \Rightarrow \kappa = \frac{g_{\mathbf{n}}(A_{\mathbf{n}} \mathbf{v}, \mathbf{v})}{g_{\mathbf{n}}(\mathbf{v}, \mathbf{v})}.
\end{equation}
Let us also note the following result from linear algebra, which is clearly applicable to the shape operators of two hypersurfaces, that are tangent to each other at a point, along a common normal vector.
\begin{prop}\label{prop:eigenValuePositiveDefinite}
    Let $A, B$ be two self-adjoint operators on a finite dimensional inner product space $\left( V, \left\langle  \right\rangle \right)$. Suppose each of the eigenvalues of $A$ strictly exceeds all the eigenvalues of $B$. Then, $A - B$ is positive definite.
\end{prop}
\ifarxiv \ifdraft {\color{blue} \fi
\begin{proof}
    Let us fix two orthonormal bases $V = \mathrm{Span}\left\langle \mathbf{u}_1,\dots, \mathbf{u}_n \right\rangle = \mathrm{Span}\left\langle \mathbf{v}_1,\dots, \mathbf{v}_n \right\rangle$, where $A \mathbf{u}_i = \kappa_i \mathbf{u}_i$ and $B \mathbf{v}_i = \tau_i \mathbf{v}_i$. By the hypothesis, $\kappa_i > \tau \coloneqq \max_{1 \le i \le n} \tau_j $ for all $i$. Now, for any $\mathbf{v} \in V$ with $\left\lVert \mathbf{v} \right\rVert \coloneqq \sqrt{\left\langle \mathbf{v}, \mathbf{v} \right\rangle} = 1$, let us write $\mathbf{v} = \sum_i a^i \mathbf{u}_i = \sum_i b^j \mathbf{v}_i$. It follows that, $1 = \left\lVert \mathbf{v} \right\rVert^2 = \sum_i \left( a^i \right)^2 = \sum_i \left( b^i \right)^2$. Then, we have 
    \[\left\langle A\mathbf{v}, \mathbf{v} \right\rangle = \sum_{i,j} a^i a^j \left\langle A \mathbf{u}_i, \mathbf{u}_j \right\rangle = \sum_{i,j} a^i a^j \left\langle \kappa_i \mathbf{u}_i, \mathbf{u}_j \right\rangle = \sum_i \left( a^i \right)^2 \kappa_i > \tau \sum_i \left( a^i \right)^2 = \tau,\]
    and similarly 
    \[\left\langle B \mathbf{v}, \mathbf{v}\right\rangle = \sum_i \left( b^i \right)^2 \tau_i \le \tau \sum_i \left( b^i \right)^2 = \tau.\] Then, we have 
    \[\left\langle (A-B)\mathbf{v}, \mathbf{v} \right\rangle = \left\langle A \mathbf{v}, \mathbf{v} \right\rangle - \left\langle B \mathbf{v}, \mathbf{v} \right\rangle = \sum_i \left( a^i \right)^2 \kappa_i - \sum_i \left( b^i \right)^2 \tau_i > 0.\]
    Since $\mathbf{v} \in V$ is an arbitrary unit vector, we have $A - B$ is positive definite.
\end{proof}
\ifdraft } \fi \fi

\subsubsection{\texorpdfstring{$N$}{N}-Geodesics}
Given a submanifold $N \subset M$, we consider the subspace of piecewise smooth paths $\mathcal{P}_N = \mathcal{P}_N([a,b]) = \{ \gamma \in \mathcal{P}([a,b]) \;|\; \gamma(a) \in N\}$ starting at $N$. Given $q \in M$, the distance from $N$ to $q$ is then defined as
    \begin{equation}\label{eq:distanceToN}
        d(N,q) \coloneqq \inf \left\{ L(\gamma) \;\middle|\; \gamma\in\mathcal{P}_N, \; \gamma(b) = q \right\}.
    \end{equation}
A piecewise $C^1$ curve $\gamma : [a,b] \rightarrow M$ in $\mathcal{P}(N)$ is called an \emph{$N$-geodesic} if $\gamma$ is a critical point of the restricted energy functional $E|_{\mathcal{P}(N)}$ in the variational sense. An $N$-geodesic $\gamma$ is called an \emph{$N$-segment} joining $N$ to $\gamma(b)$ if $\gamma$ is unit-speed, and $d(N, \gamma(b)) = b - a = L(\gamma)$. It follows from the first variation principal that an $N$-geodesic $\gamma$ has initial velocity in the normal cone of $N$. One defines the \emph{normal exponential map} $\exp^\nu: \nu(N) \rightarrow  M$ as the restriction of the exponential map to the cone bundle $\nu(N)$. Every unit-speed $N$-geodesic is then of the form $\gamma_{\mathbf{v}}(t) \coloneqq \exp^\nu(t \mathbf{v})$ for some $\mathbf{v} \in S(\nu)$. 

\subsection{Cut Locus of a Submanifold}
The \emph{cut locus of a point} $p \in M$ is the set consisting of all $q \in M$ such that there exists a minimizer from $p$ to $q$, any extension of which fails to be distance minimizing from $p$. We denote the cut locus of $p$ by $\mathrm{Cu}(p)$. Generalizing this notion to an arbitrary submanifold, we have the following definition.

\begin{defn}\label{defn:cutLocusN}
    Given a submanifold $N \subset M$, the \emph{cut locus} of $N$ consists of points $q \in M$ such that there exists an $N$-segment joining $N$ to $q$, whose extension fails to be an $N$-segment. Given $\mathbf{v} \in S(\nu)$, the \emph{cut time} of $\mathbf{v}$ is defined as 
    \begin{equation}\label{eq:cutTime}
        \rho(\mathbf{v}) = \rho_N(\mathbf{v}) \coloneqq \sup \left\{ t \;\middle|\; d(N, \gamma_{\mathbf{v}}(t \mathbf{v})) = t \right\},
    \end{equation}
    where $\gamma_{\mathbf{v}}(t) = \exp^\nu(t \mathbf{v})$ is the unique $N$-geodesic with initial velocity $\mathbf{v}$.
\end{defn}

Note that we allow $\rho$ to take the value $\infty$. The map $\rho: S(\nu ) \rightarrow  [0, \infty]$ is continuous, assuming at least one of the following holds \cite[Theorem 4.7]{BhoPra2023},
\begin{equation}\label{eq:H}\tag{$\mathsf{H}$}
    \begin{aligned} 
        &\text{(a) either $N$ is compact and $F$ is forward complete, or,} \\
        &\text{(b) $F$ is both forward and backward complete.}
    \end{aligned}
\end{equation}
Note that in \cite[Remark 3.11]{BhoPra2023} it was observed that in the absence of the hypothesis (\hyperref[eq:H]{H}), the distance to a point from a closed submanifold may not even be achieved on the submanifold.

The \emph{tangent cut locus} of $N$ is defined as 
\begin{equation}\label{eq:tangentCutLocus}
    \widetilde{\mathrm{Cu}}(N) \coloneqq \left\{ \rho(\mathbf{v}) \mathbf{v} \;\middle|\; \mathbf{v}\in S(\nu), \; \rho(\mathbf{v}) \ne \infty \right\} \subset \nu.
\end{equation}
It follows from \autoref{defn:cutLocusN} that every $x \in \mathrm{Cu}(N)$ can be written as $x = \exp^\nu(\rho(\mathbf{v}) \mathbf{v})$ for some $\mathbf{v} \in S(\nu)$. If $F$ is assumed to be forward complete, then we get the identity $\mathrm{Cu}(N) = \exp^\nu(\widetilde{\mathrm{Cu}}(N))$. To see the necessity of completeness, simply consider $N$ as a great circle in the round $2$-sphere with one of the cut points removed. In this case, the cut time is constant, $\mathrm{Cu}(N)$ consists of a point, $\widetilde{\mathrm{Cu}}(N)$ consists of two copies of $S^1$, but $\exp^\nu(\widetilde{\mathrm{Cu}}(N))$ is not defined.

\subsubsection{Separating Points and Focal Points of a Submanifold}
Let us note a useful characterization of $\mathrm{Cu}(N)$, via the notion of separating set, originally called the \emph{several geodesics set} in \cite{Wol79}. In the terminology of \cite{Bishop77}, this is also known as the set of \emph{ordinary} cut points.
\begin{defn}\label{defn:separtingSet}
    Given a submanifold $N \subset M$, a point $p \in M$ is said to be a \emph{separating point} of $N$ if there exist two distinct $N$-segments joining $N$ to $p$. The collection of all separating points of $N$ is called the \emph{separating set} of $N$, denoted $\mathrm{Se}(N)$.
\end{defn}

We have $\mathrm{Se}(N) \subset \mathrm{Cu}(N)$ \cite[Proposition 4.2]{BhoPra2023}, and in fact, under hypotheses (\hyperref[eq:H]{$\mathsf{H}$}), it follows that $\mathrm{Cu}(N) = \overline{\mathrm{Se}(N)}$ \cite[Theorem 4.8]{BhoPra2023}. We also recall the closely related notion of focal points.

\begin{defn}\label{defn:tangentFocalLocus}
    Given a submanifold $N \subset M$, a vector $\mathbf{v}\in \hat{\nu} \cap \mathcal{O}$ is said to be a \emph{tangent focal point} of $N$ if 
    \[d_{\mathbf{v}}(\exp^\nu|_{\hat{\nu} \cap \mathcal{O}}) : T_{\mathbf{v}} \left( \hat{\nu} \cap \mathcal{O} \right) = T_{\mathbf{v}} \hat{\nu} \rightarrow T_{\exp^\nu(\mathbf{v})}M\]
    is degenerate, i.e., if $\mathbf{v}$ is a critical point of $\exp^\nu|_{\hat{\nu} \cap \mathcal{O}}$. A \emph{focal point} of $N$ along an $N$-geodesic is a critical value of the normal exponential map.
\end{defn}

It follows that beyond a focal point, an $N$-geodesic can never be an $N$-segment \cite[Lemma 4.4]{BhoPra2023}. Note that although \cite{BhoPra2023} is written under the standing assumption of $F$ being forward complete, the proofs of \cite[Proposition 4.2 and Lemma 4.4]{BhoPra2023} do not require this hypothesis.

\subsection{Principal Curvatures of Small Backward Spheres}
We shall require an estimate on the principal curvatures of a small geodesic sphere. For any $p \in M$, let us define the \emph{forward} and the \emph{backward injectivity radius} at $p$ respectively as 
\begin{equation}\label{eq:injectivityRadius}
    \begin{aligned} 
        \mathrm{Inj}^{+}(p) \coloneqq \inf_{\mathbf{v}\in S(T_p M)} 
        \sup \left\{ t \;\middle|\; d(p, \exp_p(t \mathbf{v})) = t \right\},\\
        \mathrm{Inj}^{-}(p) \coloneqq \inf_{\mathbf{v} \in S(T_p M)} \sup \left\{ t \;\middle|\; d(\exp_p(t\mathbf{v}), p) = t \right\}.
    \end{aligned}
\end{equation}
It follows from the Whitehead convexity theorem \cite[Pg. 164]{Bao2000} that $\mathrm{Inj}^{\pm}(p) > 0$. We now have the following result, a proof of which can be found in \cite{Wu2016}. \ifarxiv \ifdraft {\color{blue} \fi To avoid confusion with the sign conventions, we give a complete proof in the \autoref{sec:curvatureSmallBall} \ifdraft } \fi \fi

\begin{lemma}\label{lemma:smallBallPrincipalCurvature}
    Suppose the flag curvature of $(M, F)$ is bounded from above by $\lambda$. Let $q \in M$ and choose $r > 0$ smaller than both the forward and backward injectivity radius at $q$. Then, for the backward sphere $S = S_{-}(q, r) = \left\{ x \;\middle|\; d(x, q) = r \right\}$, the principal curvatures at any point along the inward radial direction is greater than $\mathfrak{ct}_\lambda(r)$, where we have the quantity 
    \begin{equation}\label{eq:ct}
        \mathfrak{ct}_\lambda(r) = 
    \begin{cases}
        \sqrt{\lambda} \cot\left( \sqrt{\lambda} r \right), \quad \lambda > 0 \\
        \frac{1}{r}, \quad \lambda = 0 \\
        \sqrt{-\lambda} \coth\left( \sqrt{-\lambda} r \right), \quad \lambda < 0.
    \end{cases}
    \end{equation}
    In particular, as $r \to 0$, the principal curvatures of $S$, along the inward radial direction, tends to $\infty$.
\end{lemma}
\section{Geometric Tubular Neighborhood of a Submanifold} \label{sec:tubularNBD}
Let $N$ be a closed submanifold of a Finsler manifold $(M, F)$. We recall the definition first.
\begin{defn}\label{defn:geometricTubularNbd}
    Given a closed submanifold $N$ of $(M, F)$ and a positive function $\epsilon : N \rightarrow (0, \infty)$, a \emph{geometric $\epsilon$-tubular neighborhood} of $N$ is the open neighborhood of the $0$-section given as
    \[\mathcal{U} = \mathcal{U}(\epsilon) \coloneqq \left\{ \mathbf{v}\in \nu_p \;\middle|\; F(\mathbf{v}) < \epsilon(p), \; p \in N \right\} \subset \nu,\]
    such that the map $\exp^\nu|_{\mathcal{U}}$ is defined and is a homeomorphism onto the image, which, furthermore, is a diffeomorphism away from the $0$-section. Without loss of generality, we shall call the image $\exp^\nu\left( \mathcal{U} \right) \subset M$ as a tubular neighborhood of $N$ as well.
\end{defn}

When the metric is Riemannian, we refer to \cite[Theorem 5.25]{Lee18} for a detailed proof of the existence of a tubular neighborhood of a closed submanifold. For a Finsler manifold, the same was originally proved in \cite{Alves2019}, see \autoref{sec:javaloyes}. In this section, we present a different approach, which is novel even for the Riemannian setup. The idea stems from the following statement.
\begin{equation}\label{eq:innamiStatement} \tag{S}
    \begin{aligned} 
        \begin{minipage}[t]{\linewidth-3cm}
            Given $q \in M \setminus N$, with $d(N, q)$ sufficiently small, the backward sphere $S_{-}(q, d(N, q)) = \left\{ x \;\middle|\; d(x, q) = d(N, q) \right\}$ intersects $N$ at a singleton.
        \end{minipage}
    \end{aligned}
\end{equation}
 This seemingly intuitive statement (see \autoref{fig:intersectionOfCurveAndSpheres}) was used without proof in \cite[Lemma 3.2]{InnItoNagShi19}. We shall see a rigorous proof of the statement (\hyperref[eq:innamiStatement]{S}) in \autoref{cor:smallBallUniqueIntersection}, which follows immediately from \autoref{thm:cutTimePositive}. In fact, in \autoref{cor:frontIsCone}, we show how \cite[Lemma 3.2 and Theorem 3.1]{InnItoNagShi19} follow immediately from the statement (\hyperref[eq:innamiStatement]{S}). See also \autoref{sec:sphereCondition} for further discussion, including \autoref{example:ellipseParallelCurve} which gives a counterexample to the above statement in low regularity.
 \begin{figure}[H]
    \centering
    \def\svgwidth{0.3\columnwidth}
    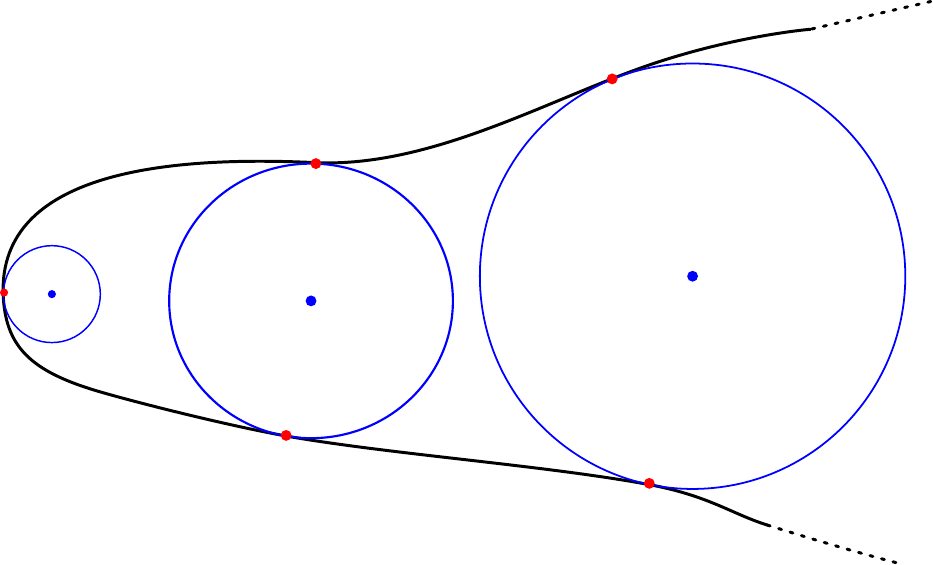

    \caption{A sufficiently small sphere intersects the submanifold at a unique point.}
    \label{fig:intersectionOfCurveAndSpheres}
\end{figure}

\subsection{Finslerian Hessian}
Let us first recall the definition of gradient of a function in the Finsler context.
\begin{defn}\label{defn:gradient}
    Given $f : M \rightarrow \mathbb{R}$, the \emph{gradient} of $f$ is defined as the vector field $\nabla f \coloneqq \mathcal{L}^{-1}(df)$, where $\mathcal{L} : TM \rightarrow T^*M$ is the Legendre transformation.
\end{defn}
Suppose, $f : M \rightarrow \mathbb{R}$ is a submersion, i.e., $df \ne 0 $ and hence, $\nabla f$ is nonvanishing. It follows that
\begin{equation}\label{eq:gradient}
    X(f) = df(X) = \mathcal{L}(\nabla f)(X) = g_{\nabla f}(\nabla f, X), \quad X \in \Gamma TM.
\end{equation}
Let us now define the hessian. Unlike the Riemannian setup, there are several non-equivalent definitions for the hessian of a function. We follow \cite{WuXin07,Ohta2021}, which is distinct from \cite{ShenBook01}. Note that if $df|_p \ne 0$, then $\nabla f \ne 0$ in a neighborhood of $p$. 
\begin{defn}\label{defn:hessian}
    Suppose $f : M \rightarrow \mathbb{R}$ is non-singular at some $p \in M$, i.e., $d f|_p \ne 0$. Then, the \emph{hessian} of $f$ at $p$ is defined as 
    \begin{equation}\label{eq:hessian}
        \textsf{Hess}_p(f)(\mathbf{x},\mathbf{y}) \coloneqq \left. \left( XY(f) - \nabla^{\nabla f}_X Y(f) \right)\right|_{p}, \quad \mathbf{x}, \mathbf{y} \in T_p M,
    \end{equation}
    where $X, Y \in \Gamma TM$ are arbitrary extensions of $\mathbf{x}, \mathbf{y}$ respectively.
\end{defn}

\begin{prop}\label{prop:hessianTensor}
    Given $f$ with $\nabla f|_p \ne 0$, the hessian of $f$ at $p$ is a well-defined symmetric $2$ tensor on $T_p M$, independent of choice of extensions. Moreover, we have the equality 
    \begin{equation}\label{eq:hessianAlt}
        \textsf{Hess}_p(f)(\mathbf{x}, \mathbf{y}) = \left. g_{\nabla f}\left( \nabla^{\nabla f}_X \nabla f, Y \right) \right|_p,
    \end{equation}
    for arbitrary extensions $X, Y$ of $\mathbf{x}, \mathbf{y}$ respectively.
\end{prop}
\begin{proof}
    Since $\nabla$ is torsion free, for any $X, Y \in \Gamma TM$ we have 
    \begin{align*}
        & \left( XY(f) - \nabla^{\nabla f}_X Y (f) \right) - \left( YX(f) - \nabla^{\nabla f}_Y X (f) \right) \\
        =& \left( XY(f) - YX(f) \right) - \left( \nabla^{\nabla f}_X Y - \nabla^{\nabla f}_Y X \right)(f) \\
        =& [X, Y](f) - [X, Y](f) = 0.
    \end{align*}
    The expression is clearly $C^\infty(M)$-linear in $X$, and hence by the above symmetry, $C^\infty(M)$-linear in $Y$ as well. Thus, $\textsf{Hess}_p(f) : T_p M \odot T_p M \rightarrow \mathbb{R}$ is a symmetric $2$-tensor, which is well-defined irrespective of the choice of local extensions.

    Next, it follows from \autoref{eq:gradient} that
    \begin{align*}
        XY(f) = X g_{\nabla f}\left( \nabla f, Y \right) &= g_{\nabla f} \left( \nabla^{\nabla f}_X \nabla f, Y \right) + g_{\nabla f}\left( \nabla f, \nabla^{\nabla f}_X Y \right) \\
        &\qquad + 2 \underbrace{C_{\nabla f}\left( \nabla^{\nabla f}_X \nabla f, \nabla f, Y \right)}_0 \\
        &=g_{\nabla f} \left( \nabla^{\nabla f}_X \nabla f, Y \right) + \nabla^{\nabla f}_X Y(f)
    \end{align*}
    Hence, we have $\textsf{Hess}_p(f)(\mathbf{x}, \mathbf{y}) = \left. g_{\nabla f}\left( \nabla^{\nabla f}_X \nabla f, Y \right) \right|_p$.
\end{proof}

Suppose $f : M \rightarrow \mathbb{R}$ is a submersion near $p \in M$ with $f(p) = 0$, and denote $P = f^{-1}(0)$ as the codimension $1$ submanifold. Then for any $X \in \Gamma TP$ we have 
\[g_{\nabla f}(\nabla f, X) = df(X) = 0.\]
Consequently, $\nabla f$ is in the normal cone of $P$ on points of $P$, similar to the Riemannian gradient. In particular, $\nabla f$ as an extension of $\mathbf{n} \coloneqq \nabla f|_p$. Consider the shape operator $A_{\mathbf{n}} : T_p P \rightarrow T_p P$ of $P$ at $p$ along $\mathbf{n}$. Then for $\mathbf{x}, \mathbf{y} \in T_p P$, we have from \autoref{prop:shapeOperator} and \autoref{prop:hessianTensor}, 
    \begin{equation}\label{eq:hessianEigenValue}
        A_{\mathbf{n}} \mathbf{x} = - \left( \nabla^{\nabla f}_X \nabla f \middle|_p \right)^{\top_{\mathbf{n}}} \Rightarrow g_{\mathbf{n}}\left( A_{\mathbf{n}} \mathbf{x}, \mathbf{y} \right) = - \left. g_{\mathbf{n}} \left( \nabla^{\nabla f}_X \nabla f, \mathbf{y} \right)\right|_p = - \textsf{Hess}_p(f)(\mathbf{x}, \mathbf{y}).
    \end{equation}

\subsection{Minimizing the Distance Function to a Point from a Submanifold}
Let us first consider the case of a hypersurface.
\begin{lemma}\label{lemma:localMinima}
    Let $P$ be a codimension $1$ submanifold of $M$ and fix a unit normal vector $\mathbf{n} \in \nu_p(P)$. For some $\epsilon > 0$, with $\epsilon \mathbf{n} \in \mathcal{O}$, consider the point $q = \exp^\nu(\epsilon \mathbf{n})$, and the distance function $f_\epsilon(x) = d(x, q)$. Then, there exists $\epsilon > 0$ sufficiently small so that $f_\epsilon|_{P}$ attains a strict local minima at $p$.
\end{lemma}
\begin{proof}
    Fix some local coordinate system $\left\{ x^1,\dots ,x^n \right\}$ given by a chart $\varphi : U \rightarrow \mathbb{R}^n$ near $p$, so that
    \[\varphi(p) = 0, \quad \partial_n|_p = \mathbf{n}, \quad  T_p P = \mathrm{Span}\left\langle \partial_1|_p, \dots , \partial_{n-1}|_p \right\rangle,\]
    where we have denoted the vector fields $\partial_i \coloneqq d\varphi^{-1}\left( \frac{\partial}{\partial x^i} \right)$. We consider $U$ to be relatively compact, i.e., $\overline{U}$ is compact, in order to assert the following.
    \begin{enumerate}[label=\roman*)]
        \item There is an upper bound on the flag curvature of $F$ on $U$.
        \item The principal curvatures at each point of $U \cap P$ in the (unique) unit normal direction are bounded from above by some $\kappa > 0$.
        \item There is a number, say, $\delta > 0$ so that for any $q \in U$, the backward ball $B_{-}(q, \delta) = \left\{ x \;\middle|\; d(x, q) < \delta \right\}$ is strongly geodesically convex, i.e., for any $q_1, q_2 \in B_{-}(q, \delta)$ there exists a unique minimizer joining $q_1$ to $q_2$, which, furthermore, lies completely inside $B_{-}(q, \delta)$. The existence of such $\delta$ follows from the Whitehead convexity theorem \cite[Pg. 164]{Bao2000}.
    \end{enumerate}
    Next, we choose $\epsilon > 0$ sufficiently small satisfying the following.
    \begin{enumerate}[label=\alph*)]
        \item Firstly, we assume that geodesic $\gamma_{\mathbf{v}}(t) = \exp^\nu(t \mathbf{v})$ is a minimizer joining $p$ to $q \coloneqq \exp^\nu(\epsilon \mathbf{v})$, and the backward ball $B_{-}(q, \epsilon) = \left\{ x \;\middle|\; d(x, q) < \epsilon \right\}$ is contained in $U$.
        \item We assume $\epsilon < \frac{\delta}{2}$. Then, for any $x \in U \setminus \left\{ q \right\}$, there exists a unique minimizer joining $x$ to $q$, and consequently, the distance function $f(x) = f_\epsilon(x) = d(x, q)$ is smooth on $U \setminus \left\{ q \right\}$ \cite{SaTa16}.
        \item As the flag curvature of $F$ is upper bounded in $U$, by \autoref{lemma:smallBallPrincipalCurvature}, we assume $\epsilon > 0$ small enough so that all the principal curvatures of the backward sphere 
        \[S \coloneqq S_{-}(q, \epsilon) = \left\{ x \;\middle|\; d(x, q) = \epsilon \right\}\]
        along the inward radial directions strictly exceeds $\kappa$.
    \end{enumerate}
    
    Now, note that $p \in S \cap P$. Furthermore, $\mathbf{n}$ is the inward radial direction of $S$ at $p$, and hence, $T_p P = \left\langle \mathbf{n} \right\rangle^{\perp_{g_{\mathbf{n}}}} = T_p S$. By the implicit function theorem, we have a smaller neighborhood $p \in U_1 \subset U$, and two smooth functions $h^P, h^S : \mathbb{R}^{n-1} \rightarrow \mathbb{R}$ defined near $0$, so that $\varphi(P \cap U_1)$ and $\varphi(S \cap U_1)$ are given as the graphs of $h^P$ and $h^S$ respectively (see \autoref{fig:graphLocalMinima}). Since by choice $T_p S = T_p P = \mathrm{Span}\left\langle \partial_i \middle|_p, 1\le i \le n-1 \right\rangle$, it follows that $h^P(0) = 0 = h^S(0)$, and $\left.\frac{\partial h^P}{\partial x^i}\right|_0 = 0 = \left.\frac{\partial h^S}{\partial x^i}\right|_0$. That is, $dh^P|_0 = 0 = dh^S|_0$. Denoting the projection onto the first $n-1$ components as $\pi : \mathbb{R}^n \rightarrow \mathbb{R}^{n-1}$, we get two functions 
    \[\psi^P \coloneqq x^n - h^P \circ \pi \circ \varphi, \quad \psi^S \coloneqq x^n - h^S \circ \pi \circ \varphi\]
    defined on $U_1$, so that $P \cap U_1 = \left( \psi^P \right)^{-1}(0)$ and $S \cap U_1 = \left( \psi^S \right)^{-1}(0)$. Note that $d\psi^P|_p = dx^n = d\psi^S|_p$. Denote, $\psi = \psi^S - \psi^P$, so that $d \psi|_p = 0$. We cannot directly compute $\textsf{Hess}_p(\psi)$ since it is not defined.
    \begin{figure}[H]
        \centering
    \def\svgwidth{0.4\columnwidth}
\begingroup%
  \makeatletter%
  \providecommand\color[2][]{%
    \errmessage{(Inkscape) Color is used for the text in Inkscape, but the package 'color.sty' is not loaded}%
    \renewcommand\color[2][]{}%
  }%
  \providecommand\transparent[1]{%
    \errmessage{(Inkscape) Transparency is used (non-zero) for the text in Inkscape, but the package 'transparent.sty' is not loaded}%
    \renewcommand\transparent[1]{}%
  }%
  \providecommand\rotatebox[2]{#2}%
  \newcommand*\fsize{\dimexpr\f@size pt\relax}%
  \newcommand*\lineheight[1]{\fontsize{\fsize}{#1\fsize}\selectfont}%
  \ifx\svgwidth\undefined%
    \setlength{\unitlength}{428.1303164bp}%
    \ifx\svgscale\undefined%
      \relax%
    \else%
      \setlength{\unitlength}{\unitlength * \real{\svgscale}}%
    \fi%
  \else%
    \setlength{\unitlength}{\svgwidth}%
  \fi%
  \global\let\svgwidth\undefined%
  \global\let\svgscale\undefined%
  \makeatother%
  \begin{picture}(1,0.6117029)%
    \lineheight{1}%
    \setlength\tabcolsep{0pt}%
    \put(0,0){\includegraphics[width=\unitlength,page=1]{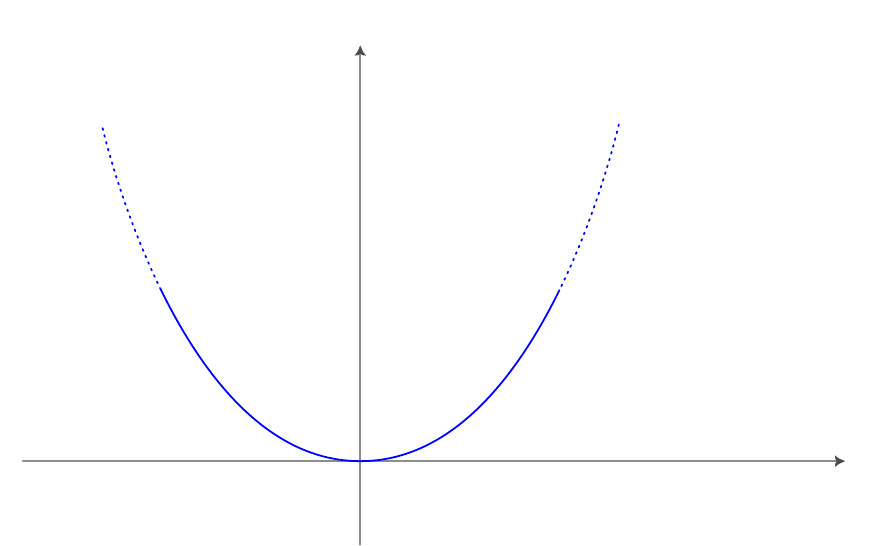}}%
    \put(0.35836092,0.58094809){\color[rgb]{0,0,0}\transparent{0.98000002}\makebox(0,0)[lt]{\lineheight{0}\smash{\begin{tabular}[t]{l}$x^n$\end{tabular}}}}%
    \put(0.69821048,0.40473152){\color[rgb]{0,0,0}\transparent{0.98000002}\makebox(0,0)[lt]{\lineheight{0}\smash{\begin{tabular}[t]{l}$S$\end{tabular}}}}%
    \put(0,0){\includegraphics[width=\unitlength,page=2]{local_graph.pdf}}%
    \put(0.74082855,0.22227734){\color[rgb]{0,0,0}\transparent{0.98000002}\makebox(0,0)[lt]{\lineheight{0}\smash{\begin{tabular}[t]{l}$P$\end{tabular}}}}%
    \put(0.51934051,0.05130267){\color[rgb]{0,0,0}\transparent{0.98000002}\makebox(0,0)[lt]{\lineheight{0}\smash{\begin{tabular}[t]{l}$x^1, \dots, x^{n-1}$\end{tabular}}}}%
    \put(0,0){\includegraphics[width=\unitlength,page=3]{local_graph.pdf}}%
    \put(0.35090811,0.0597553){\color[rgb]{0,0,0}\makebox(0,0)[lt]{\lineheight{1.25}\smash{\begin{tabular}[t]{l}$p$\end{tabular}}}}%
  \end{picture}%
\endgroup%

        \caption{In local coordinates near $p$, the hypersurfaces $S$ and $P$ are graphs of $h^S$ and $h^P$.}
        \label{fig:graphLocalMinima}
    \end{figure}
    
    Now, for $\mathbf{x}, \mathbf{y} \in T_p P = T_p S$, choose some local extensions $X, Y \in \Gamma TU_1$. Then, we compute
    \begin{align*}
        \textsf{Hess}_p(\psi^S)(\mathbf{x}, \mathbf{y}) - \textsf{Hess}_p(\psi^P)(\mathbf{x}, \mathbf{y}) &= \left. \left( XY\left( \psi^S - \psi^P \right) - \nabla^{\mathbf{n}}_X Y \left( \psi^S - \psi^P \right) \right) \right|_p \\
        &= XY(\psi)|_p - d \psi|_p \left( \nabla^{\mathbf{n}}_X Y \middle|_p \right) \\
        &= XY(\psi)|_p.
    \end{align*}
    In particular, for $\partial_i, \partial_j$ with $1 \le i,j \le n-1$, we have
    \begin{align*}
        \textsf{Hess}_p(\psi^S)(\partial_i, \partial_j) - \textsf{Hess}_p(\psi^P)(\partial_i, \partial_j) = \partial_i \partial_j(\psi)|_p = -\left.\frac{\partial^2}{\partial x^i \partial x^j} \right|_0 (h^S - h^P).
    \end{align*}
    In view of \autoref{eq:hessianEigenValue}, the hessian matrix of $h^S - h^P$ at $0$ is then given as
    \[\begin{pmatrix}
        \left.\frac{\partial^2}{\partial x^i \partial x^j} \right|_0 (h^S - h^P)
    \end{pmatrix}_{(n-1)\times (n-1)} = 
    \begin{pmatrix}
        g_{\mathbf{n}}\left( A_{\mathbf{n}}^S \partial_i, \partial_j \right)
    \end{pmatrix} - \begin{pmatrix}
        g_{\mathbf{n}}\left( A_{\mathbf{n}}^P \partial_i, \partial_j \right)
    \end{pmatrix},\]
    where $A_{\mathbf{n}}^S, A_{\mathbf{n}}^P$ are the shape operators associated to $S, P$ respectively, in the direction of $\mathbf{n}$. Now, by our choice, each of eigenvalues of $A_{\mathbf{n}}^S$ strictly dominates the eigenvalues of $A_{\mathbf{n}}^P$. By an application of \autoref{prop:eigenValuePositiveDefinite}, it then follows that the hessian matrix of $h^S - h^P$ is positive definite at $0$. Hence, $h^S - h^P$ attains a strict local minima at $0$, i.e., $h^S(0) = 0 = h^P(0)$ and $h^S > h^P$ in a deleted neighborhood of $0$. But then for some $p \in O \subset U_1$, we have $(S \cap O) \cap (P \cap O) = \left\{ p \right\}$, and $S \cap O$ lies completely on one side of $P\cap O$. In other words, the distance function $f|_{P} = f_\epsilon|_P$ then attains a strict local minima at $p \in P$.
\end{proof}

The next result proves that for $\epsilon$ even smaller, we can get a \emph{global} minima for the distance function for a closed submanifold with arbitrary codimension. The argument presented should be compared to \cite[Lemma 2.3]{Xu2015} and \cite[Lemma 4.18]{BhoPra2023}.

\begin{lemma}\label{lemma:globalMinima}
    Let $P$ be a closed submanifold of $(M, F)$ of positive codimension. Fix some unit normal vector $\mathbf{n} \in \nu_p(P)$. Suppose, for some $\epsilon > 0$, with $\epsilon\mathbf{n} \in \mathcal{O}$, the distance function $f_\epsilon(x) = d(x, \exp^\nu(\epsilon\mathbf{n}))$ attains a strict local minima on $P$ at $p$. Then, for some smaller $\epsilon$, we have $f_\epsilon|_P$ attains a strict \emph{global} minima at $p$ as well.
\end{lemma}
\begin{proof}
    Since $f_\epsilon|_P$ attains a strict local minima at $P$, it is immediate that the geodesic $\gamma_{\mathbf{n}}(t) = \exp^\nu(t \mathbf{n})$ is a minimizer joining $p$ to $q \coloneqq \exp^\nu(\epsilon \mathbf{n})$. For any $0 < \epsilon^\prime < \epsilon$ consider the point $q^\prime = \exp^\nu(\epsilon^\prime \mathbf{n})$ on $\gamma_{\mathbf{n}}$, and note that $d(q^\prime, q) = \epsilon - \epsilon^\prime$. Now, we always have the inclusion $B_{-}(q^\prime, \epsilon^\prime) \subset B_{-}(q, \epsilon)$ (see \autoref{fig:globalMinima}), since for any $x$ with $d(x, q^\prime) < \epsilon^\prime$ it follows that
    \[d(x, q) \le d(x, q^\prime) + d(q^\prime, q) < \epsilon^\prime + (\epsilon - \epsilon^\prime) = \epsilon.\]
    Observe that for the distance function $f_{\epsilon^\prime}(x) = d(x, q^\prime)$, we see that $f_{\epsilon^\prime}|_P$ still attains a strict local minima at $p$. Indeed, for any $x \ne p$ sufficiently close to $p$ we have 
    \[f_{\epsilon^\prime}(x) = d(x,q^\prime) \ge d(x, q) - d(q^\prime, q) > d(p, q) - d(q^\prime, q) = d(p, q^\prime) = f_{\epsilon^\prime}(p).\]
    \begin{figure}[H]
        \centering
    \def\svgwidth{0.6\columnwidth}
    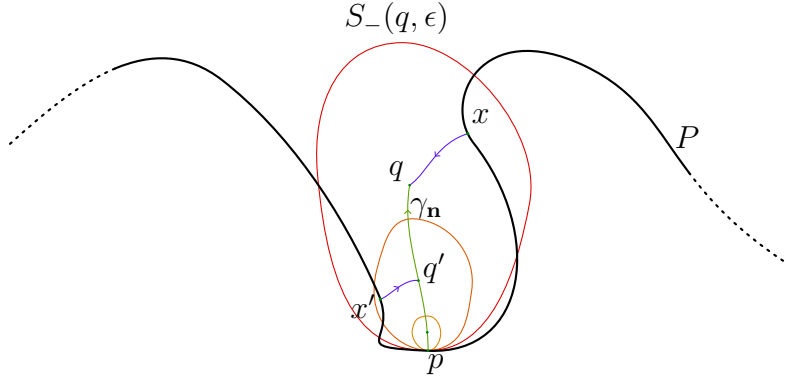

        \caption{For the point $q^\prime$ on $\gamma_{\mathbf{n}}$, the distance from $P$ is achieved at $x^\prime$.}
        \label{fig:globalMinima}
    \end{figure}
    Without loss of generality, we can take $\epsilon > 0$ small enough so that the closed backward ball $\overline{B_{-}(q, \epsilon)}$ is compact, which follows without any completeness assumption on $F$. Let us denote, 
    \[K \coloneqq \overline{B_{-}(q, \epsilon)} \cap P,\]
    which is clearly compact as $P$ is assumed to be closed. Now, $f_{\epsilon^\prime}|_P$ attains a global minima at $p$ precisely when $d(P, q^\prime) = \epsilon^\prime = d(p, q^\prime)$. Otherwise, we have $d(P, q^\prime) = \inf_{x \in P} d(x, q^\prime) < d(p, q^\prime)$. Choose $p_i \in P$ so that $d(P, q^\prime) \le d(p_i, q^\prime) < d(p, q^\prime) = \epsilon^\prime$ and $d(p_i, q^\prime) \rightarrow d(P, q^\prime)$. Now, $p_i \in \overline{B_{-}(q, \epsilon)}$, and thus, $p_i \in K$. Hence, we have a convergent subsequence $p_{i_j} \rightarrow p_0 \in K \subset P$. Clearly, 
    \[d(P, q^\prime) = \lim d(p_i, q^\prime) = \lim d(p_{i_j}, q^\prime) = d(p_0, q^\prime).\]
    In other words, $f_{\epsilon^\prime}|_P$ always attains a global minima at some $p_0 \in K$, possibly different from $p$, for any $0 < \epsilon^\prime < \epsilon$ (see \autoref{fig:globalMinima}).

    Since $f_\epsilon|_P$ attains a strict local minima at $p$, we have some open neighborhood $p \in O \subset P$, such that $f_{\epsilon}|_{P \cap O}$ attains a strict global minima at $p$. We assume that $O \subset B_{-}(q, \epsilon) \cap P \subset K$. Let us now choose some $\epsilon_i \rightarrow 0$ with $\epsilon_i < \epsilon$, and denote $q_i = \exp^\nu(\epsilon_i \mathbf{n})$. Clearly, $q_i \rightarrow p$. Denote the associated distance functions as $f_i(x) \coloneqq d(x, q_i)$. Now, as argued in the previous paragraph, $f_i|_P$ always attains a global minima in $K$. If possible, suppose, it only attains a global minima in $K \setminus O$. That is, we have  $f_{i}|_P$ attains a global minima at $y_i \in K \setminus O$. Since $K \setminus O$ is again compact, passing to subsequence, we have $y_i \rightarrow y \in K \setminus O$. Now, for any $x \in K$ we have $d(y_i, q_i) = f_i(y_i) \le f_i(x) = d(x, q_i)$. Taking $i\rightarrow \infty$ we have $d(y, p) \le d(x, p)$. Since $x \in K$ is arbitrary, taking $x = p$ we have $d(y, p) \le 0 \Rightarrow d(y, p) = 0 \Rightarrow y = p$, which contradicts $y \not \in O$. Hence, for some $\epsilon_{i_0}$ we must have that $f_{i_0}|_P$ attains a global minima, say, $y_{i_0}$ inside $O$. But then $p = y_{i_0}$ as $f_{i_0}|_{P \cap O}$ attains a strict global minima at $p$. This concludes the proof.
\end{proof}

The following lemma shows that if in some normal direction to an arbitrary submanifold the geodesic is globally distance minimizing for some time, then the same holds true in a neighborhood of that direction.

\begin{lemma}\label{lemma:uniformPositive}
    Let $N$ be a submanifold of $(M, F)$. Suppose, for some unit normal vector $\mathbf{v}_0 \in \nu_p$, we have the cut time $\rho(\mathbf{v}_0) > 0$. Then, there exists some $\epsilon > 0$ and a neighborhood $\mathbf{v}_0 \in V \subset S(\nu)$, such that $\rho(\mathbf{v}) \ge \epsilon$ for all $\mathbf{v} \in V$.
\end{lemma}
\begin{proof}
    Without assuming $F$ to be complete, one can fix some $\delta > 0$ sufficiently small so that the closed forward ball $B = \overline{B_{+}(p, \delta)} = \left\{ x \;\middle|\; d(p, x) \le \delta \right\}$ is compact. Since the exponential map $\exp : \mathcal{O} \subset TM \rightarrow M$ is continuous in an open neighborhood $\mathcal{O}$ of $TM$ containing the $0$-section, we can get some $\epsilon > 0$ such that $\exp_q(t \mathbf{v})$ is defined for all time $0 \le t \le \epsilon$, and for all $\mathbf{v} \in S(\nu_q)$ with $q \in B$. We may further assume that $0 < \epsilon < \rho(\mathbf{v}_0)$. Set $U = \left\{ \mathbf{v} \in S(\nu_q) \;\middle|\; q \in B_{+}(p, \delta) \cap N \right\}$, which is clearly an open neighborhood of $\mathbf{v}_0$ in $S(\nu)$. If possible, suppose we have a sequence $\mathbf{v}_i \in U$ converging to $\mathbf{v}_0$, with $\lim \inf \rho(\mathbf{v}_i) < \epsilon$. Set $\epsilon_i \coloneqq \rho(\mathbf{v}_i)$, and passing to a subsequence, assume that $\epsilon_i < \epsilon$ for all $i$. Consider the geodesics $\gamma_i : [0, \epsilon] \rightarrow M$ given by $\gamma_i(t) = \exp^\nu(t \mathbf{v}_i)$, and set $q_i \coloneqq \gamma_i(\epsilon)$. By the definition of cut time, $\gamma_i|_{[0, \epsilon_i]}$ is an $N$-segment, which fails to be so beyond that point. In particular, $d(N, q_i) < \rho(\mathbf{v}_i) = \epsilon_i$. Now, $q \coloneqq \exp^\nu(\epsilon \mathbf{v}) = \lim_i \exp^\nu(\epsilon \mathbf{v}_i) = \lim \gamma_i(\epsilon) = \lim q_i$. Hence, in the limit we get $d(N, q) = \lim d(N, q_i) \le \lim \epsilon_i = 0$. This contradicts the fact $q = \exp^\nu(\epsilon \mathbf{v})$ and $\rho(\mathbf{v}) > \epsilon > 0$. Thus, for some possibly smaller neighborhood $\mathbf{v}_0 \in V \subset U \subset S(\nu)$, we have $\rho(\mathbf{v}) \ge \epsilon$ for all $\mathbf{v} \in V$. This concludes the proof.
\end{proof}

We are now in a position to prove the main theorem of this section.
\begin{theorem}\label{thm:cutTimePositive}
    Suppose $N$ is a closed (not necessarily compact) submanifold of a Finsler manifold $(M, F)$. Then, the cut time map $\rho : S(\nu(N)) \rightarrow [0, \infty]$ is always positive. Moreover, for any $\mathbf{n} \in S(\nu)$, there exists some $\epsilon > 0$ and a neighborhood $\mathbf{n} \in V \subset S(\nu)$, such that $\rho(\mathbf{v}) \ge \epsilon$ for all $\mathbf{v} \in V$.
\end{theorem}
\begin{proof}
    Pick some $\mathbf{n} \in S(\nu_p(N))$, and consider the geodesic $\gamma_{\mathbf{n}}(t) = \exp^\nu(t \mathbf{n})$. Now, $\mathcal{L}(\mathbf{n})$ is a $1$-form which annihilates $T_p N$. In particular, we have a function $h$ defined on some open neighborhood $p \in W \subset M$, so that $h(p) = 0$ and $dh$ annihilates $TN$. Moreover, $\tilde{\mathbf{n}} = \nabla h \in S(\nu)$. Define $P = h^{-1}(0)$. Clearly, $P$ is a codimension $1$ submanifold, and $N \cap W \subset P$. Furthermore, $\mathbf{n} \in S(\nu_p(P))$, and hence $\gamma_{\mathbf{n}}$ is a $P$-geodesic as well. By \autoref{lemma:localMinima}, we have some $\epsilon > 0$ so that the distance function $f_\epsilon(x) = d(x, \gamma_{\mathbf{n}}(\epsilon))$ attains a strict local minima on $P$ at $p$. Since $N \cap W \subset P$, it follows that $f_\epsilon|_{N \cap W}$ attains a strict local minima at $p$ as well. As $N \cap W$ is open in $N$, we have $f_\epsilon|_N$ attains a strict local minima at $p$. But then by \autoref{lemma:globalMinima}, for some $0 < \epsilon^\prime < \epsilon$, we have $f_{\epsilon^\prime}|_N$ attains a strict \emph{global} minima at $p$. In other words, $d(N, \gamma_{\mathbf{n}}(\epsilon^\prime)) = \epsilon^\prime$. But then by the definition of cut time (\autoref{eq:cutTime}), we have $\rho(\mathbf{n}) \ge \epsilon^\prime > 0$. We conclude the proof by applying \autoref{lemma:uniformPositive}.
\end{proof}

Note that the above theorem holds true without any completeness assumption on $F$. As an immediate corollary, we get the following.
\begin{corollary}\label{cor:cutLocusDisjoint}
    Let $N$ be a closed submanifold of a Finsler manifold $(M, F)$. Then, $N \cap \mathrm{Cu}(N) = \emptyset$.
\end{corollary}
\begin{proof}
    If $x \in N \cap \mathrm{Cu}(N)$, then $x = \exp^\nu(\rho(\mathbf{v}) \mathbf{v})$ for some $\mathbf{v}\in S(\nu)$. But then $x \in N \Rightarrow d(N, x) = 0 \Rightarrow \rho(\mathbf{v}) = 0$, contradicting \autoref{thm:cutTimePositive}. Hence, $N \cap \mathrm{Cu}(N) = \emptyset$.
\end{proof}

Let us also note the following easy consequence.

\begin{corollary}\label{cor:smoothnessOfDistanceSquare}
    Let $N$ be closed submanifold of a forward complete Finsler manifold $(M,F)$, and furthermore, hypothesis (\hyperref[eq:H]{H}) holds. Then, the distance squared function $f(x) \coloneqq d(N, x)^2$ is $C^1$-smooth on an open neighborhood $U$ of $N$, and is $C^\infty$ on $U \setminus N$.
\end{corollary}
\begin{proof}
    By \cite[Theorem 4.12]{BhoPra2023}, we have $f$ is $C^1$ on $M \setminus \mathrm{Se}(N)$ and $C^\infty$ on $M \setminus \left( N \cup \mathrm{Cu}(N) \right)$. Since $\mathrm{Cu}(N)$ is closed by \cite[Theorem 4.8]{BhoPra2023}, we have from \autoref{cor:cutLocusDisjoint} that $N \setminus \mathrm{Cu}(N)$ is an open neighborhood of $N$. The proof follows.
\end{proof}

\subsection{Existence of Tubular Neighborhood}
As an application of \autoref{thm:cutTimePositive}, we can now prove the existence of a geometric tubular neighborhood for an arbitrary closed submanifold in a Finsler manifold.

\begin{theorem}\label{thm:tubularNBD}
    Let $N$ be a closed submanifold of a Finsler manifold $(M, F)$. Then, there exists a smooth function $\epsilon : N \rightarrow (0, \infty)$ such that $N$ admits a geometric $\epsilon$-tubular neighborhood. If $N$ is assumed to be compact, then we can take $\epsilon > 0$ as constant, and moreover, for some possibly smaller $\epsilon > 0$, the image of the $\epsilon$-tubular neighborhood under $\exp^\nu$ can then be identified with the set $\left\{ x \;\middle|\; d(N, x) < \epsilon \right\}$.
\end{theorem}
\begin{proof}
    Let $K \subset N$ be a compact set. For any $\mathbf{v} \in S(\nu_p)$ with $p \in K$, by \autoref{thm:cutTimePositive}, we can get an $\epsilon_{\mathbf{v}} > 0$ and an open neighborhood $\mathbf{v} \in U_{\mathbf{v}} \subset S(\nu)$, such that $\rho(\mathbf{u}) \ge \epsilon_{\mathbf{v}}$ for all $\mathbf{u} \in U_{\mathbf{v}}$. Set $V_{\mathbf{v}} \coloneqq \pi \left( U_{\mathbf{v}} \right)$, where $\pi : S(\nu) \rightarrow N$ is the projection map. Clearly, $V_{\mathbf{v}} \subset N$ is open, and we have an open covering $\left\{ V_{\mathbf{v}} \;\middle|\; \mathbf{v} \in S(\nu) \cap \pi^{-1}(K) \right\}$. Since $K$ is compact, we have a finite covering, say, $\left\{ V_{\mathbf{v}_j} \right\}_{j=1}^{k_0}$. Set, $\epsilon_K = \min_j \epsilon_{\mathbf{v}_j} > 0$. Then we have 
    \[\rho(\mathbf{v}) \ge \epsilon_K , \quad \mathbf{v} \in S(\nu) \cap \pi^{-1}(K).\]
    Now, consider a compact exhaustion of $N$, i.e., get countably many open sets $B_i \subset N$, such that 
    \[N = \bigcup_{i \ge 1} B_i, \quad \overline{B_i} \subset B_{i+1}, \quad \text{$\overline{B_i}$ is compact.}\]
    By the previous paragraph, for each such $B_i$, we can get some $\epsilon_i \coloneqq \epsilon_{\overline{B_i}} > 0$. By a standard argument using partition of unity, we can now get a smooth function $\epsilon : N \rightarrow (0, \infty)$ such that, $\epsilon|_{\overline{B_i}} \le \epsilon_i$. Then, $\rho(\mathbf{v}) \ge \epsilon(p)$ for any $\mathbf{v} \in S(\nu_p)$. Consider the set
    \[\mathcal{U} \coloneqq \cup_{p \in N} \left\{ t \mathbf{v} \;\middle|\; 0 \le t < \epsilon(p), \; \mathbf{v} \in S(\nu_p) \right\} \subset \nu.\]
    Note that $\mathcal{U} \subset \mathcal{O}$, the domain of the exponential map. Since $\epsilon$ is continuous, we see that $\mathcal{U}$ is an open neighborhood of the $0$-section in $\nu$. As $\mathcal{U}$ avoids the tangential cut locus $\widetilde{\mathrm{Cu}}(N)$, its image avoids cut locus $\mathrm{Cu}(N)$. In particular, $\exp^\nu(\mathcal{U})$ does not intersect $\mathrm{Se}(N)$, and consequently, $\exp^\nu|_{\mathcal{U}}$ is injective. It follows from the invariance of domain that $\exp^\nu|_{\mathcal{U}}$ is then a homeomorphism onto its image. Denote, $\widehat{\mathcal{U}} \coloneqq \mathcal{U} \setminus \mathbf{0}$. For any $\mathbf{v} \in S(\nu_p)$ and $0 < t < \epsilon(p)$, the $N$-geodesic $\gamma_{\mathbf{v}}$ is an $N$-segment in $[0, t]$. In particular, $t\mathbf{v} \in \widehat{\mathcal{U}}$ is not a tangent focal point of $N$. Hence, by the implicit function theorem, $\exp^\nu|_{\widehat{\mathcal{U}}}$ is a local diffeomorphism at $t\mathbf{v}$. But then $\exp^\nu|_{\widehat{\mathcal{U}}}$ is a diffeomorphism, as it is shown to be injective. Thus, $\mathcal{U}$ is a geometric $\epsilon$-tubular neighborhood of $N$.
    
    If $N$ is compact, we can take $K = N$ and the constant $\epsilon = \epsilon_N > 0$ so that $\rho(\mathbf{v}) \ge \epsilon$ for all $\mathbf{v} \in S(\nu)$, and get the same tubular neighborhood as above. Let us now identify the image. Firstly, as in the proof of \autoref{lemma:localMinima}, by the Whitehead convexity theorem, we may assume $\epsilon > 0$ to be so small that for any $p \in N$, the forward ball $B_{+}(p, \epsilon) = \left\{ x \;\middle|\; d(p, x) < \epsilon \right\}$ is strongly geodesically convex. Now, for any $\mathbf{v} \in S(\nu)$ and for any $t < \epsilon \le \rho(\mathbf{v})$, we have $d(N, \exp^\nu(t \mathbf{v})) = t < \epsilon$. Hence, $\exp^\nu(\mathcal{U}) \subset \mathcal{V}_\epsilon \coloneqq \left\{ x \;\middle|\; d(N, x) < \epsilon \right\}$. For the converse, choose some $x \in M$ with $d(N, x) < \epsilon$. Then, there is a sequence $p_i \in N$ such that $d(N, x) = \lim d(p_i, x)$. As $N$ is compact, we may assume that $p_i \rightarrow p \in N$. Then, $d(p,x) = \lim d(p_i, x) = d(N, x) < \epsilon$. By the convexity, there exists a (unique) minimizer, say, $\gamma$ joining $p$ to $x$. As $\gamma$ has the length $d(N, x)$, by the first variational principal, we see that $\gamma$ is an $N$-segment. Hence, we can write, $\gamma(t) = \exp^\nu(t \mathbf{v})$, for $\mathbf{v} = \dot \gamma(0) \in S(\nu_p)$. But then $x = \exp^\nu(d(N,x)\mathbf{v}) \in \exp^\nu\left( \mathcal{U} \right)$, as $d(N, x) < \epsilon$. Thus, $\mathcal{V}_\epsilon \subset \exp^\nu(\mathcal{U})$. The claim then follows.
\end{proof}

Rephrasing the second part of \autoref{thm:tubularNBD}, we get the following.
\begin{corollary}\label{cor:smallBallUniqueIntersection}
    Suppose $N$ is a compact submanifold of a Finsler manifold $(M, F)$. Then, there exists an $\epsilon > 0$ such that for any $q \in M$ satisfying $d(N, q) < \epsilon$ the backward sphere $S_{-}(q, d(N,q))$ intersects $N$ in a unique point.
\end{corollary}

As mentioned earlier, the above corollary proves the statement (\hyperref[eq:innamiStatement]{S}). Motivated by \autoref{eq:injectivityRadius}, we can define the \emph{forward injectivity radius} of $N$ as 
\begin{equation}\label{eq:injectivityRadiusSubmanifold}
    \mathrm{Inj}^{+}(N) \coloneqq \inf_{\mathbf{v} \in S(\nu)} \rho(\mathbf{v}) = \inf_{\mathbf{v} \in S(\nu)}  \sup \left\{ t \;\middle|\; d(N, \exp^\nu(t\mathbf{v})) = t \right\}.
\end{equation}
It is immediate that $d(N, \mathrm{Cu}(N)) \ge \mathrm{Inj}^{+}(N)$. Then, we can restate the above as follows. 
\begin{corollary}\label{cor:injectivityRadiusPositive}
    Given a compact submanifold $N$ of a Finsler manifold $(M, F)$, we have $\mathrm{Inj}^{+}(N)$ is positive, and consequently, $d(N, \mathrm{Cu}(N)) > 0$.
\end{corollary}

The next corollary should be compared to \cite[Theorem 3.1 and Lemma 3.2]{InnItoNagShi19}, which state the backward version of the same, under the stronger hypothesis of both directional completeness.
\begin{corollary}\label{cor:frontIsCone}
    Let $N$ be a compact submanifold of a forward complete Finsler manifold $(M, F)$. Then, for any $p \in N$, the set 
    \[\mathcal{F} = \mathcal{F}_p = \left\{ q \in M \setminus \mathrm{Cu}(N) \;\middle|\; \text{there exists a unique $N$-segment, which joins $p$ to $q$} \right\}\]
    is a topological cone on a smooth sphere of dimension $\codim N - 1$, with $p$ as the cone point. Furthermore, for any $0 < r < \mathrm{Inj}^{+}(N)$, the set $\mathcal{S}_{+}(N, r) = \left\{ x \;\middle|\; d(N, x) = r \right\}$ is diffeomorphic to $S(\nu)$, and $\mathcal{F} \cap \mathcal{S}_{+}(N, r)$ is diffeomorphic to $S(\nu_p)$.
\end{corollary}
\begin{proof}
    Consider the set 
    \[\mathcal{C} = \mathcal{C}_p = \left\{ t\mathbf{v} \;\middle|\; \mathbf{v}\in S(\nu_p), \; 0 \le t < \rho(\mathbf{v}) \right\}.\]
    Since the hypothesis (\hyperref[eq:H]{H}) is satisfied, $\rho$ is continuous, and by \autoref{thm:cutTimePositive} it is strictly positive. Hence, $\mathcal{C}$ is a topological cone on the smooth sphere $S(\nu_p)$ of dimension $\codim N - 1$. But as observed in the proof of \autoref{thm:tubularNBD}, we see that $\exp^\nu$ restricts to a homeomorphism on $\mathcal{C}$ with image $\mathcal{F}$.
    
    For any $r > 0$, the set $\mathfrak{S}(\nu, r) = \left\{ r \mathbf{v} \;\middle|\; \mathbf{v} \in S(\nu) \right\}$ is clearly diffeomorphic to $S(\nu)$. As noted in \autoref{cor:injectivityRadiusPositive}, we have $\mathrm{Inj}^{+}(N) > 0$. Hence, the same argument as in \autoref{thm:tubularNBD} shows that for any $0 < r < \mathrm{Inj}^{+}(N)$, the map $\exp^\nu$ is a diffeomorphism restricted to $\mathfrak{S}(\nu, r)$ with image $\mathcal{S}_{+}(N, r)$. Clearly, $\mathcal{F} \cap \mathcal{S}_{+}(N, r)$ is the image of the set $\mathfrak{S}_p(\nu, r) = \left\{ r \mathbf{v} \;\middle|\; \mathbf{v} \in S(\nu_p) \right\}$ under $\exp^\nu$. This concludes the proof.
\end{proof}

Note that the above theorem cannot be true in the absence of completeness. Indeed, consider $M = \mathbb{R}^2 \setminus {(1,0)}$, and take $N = \left\{ p = (0,0) \right\}$. Then, we have $\mathcal{F}_p = \mathbb{R}^2 \setminus \left\{ (x,0) \;\middle|\; x \ge 1 \right\}$, which clearly is not a topological cone. In this example, $\rho$ fails to be continuous.

\section{Relation to the Existing Literature}\label{sec:existingLit}
In this section, we discuss the how the present work relates to the existing literature, specifically to \cite{Alves2019} and \cite{InnItoNagShi19}.

\subsection{Relation with the Article \texorpdfstring{\cite{Alves2019}}{} by Alves and Javaloyes}\label{sec:javaloyes} As noted earlier, the main results in this article already appeared in \cite{Alves2019}. In particular, in \cite[Theorem 3.1]{Alves2019} the authors claimed that given any $\mathbf{v}_0 \in S(\nu)$, there exists some $\epsilon > 0$ and some neighborhood $\mathbf{v}_0 \in U \subset S(\nu)$, such that $\rho(\mathbf{u}) > \epsilon$ for all $\mathbf{u} \in U$. Their proof strategy there was the following. Given any submanifold $P$ of $M$, they claimed to produce codimension $1$ orientable submanifolds $\tilde{P}_{\mathbf{v}}$, parametrized by $\mathbf{v} \in S(\nu)$, such that $P \subset \tilde{P}_{\mathbf{v}}$, and moreover, $\mathbf{v}$ is in the normal cone of $\tilde{P}_{\mathbf{v}}$. Then, the problem of proving that \emph{the cut time of $\mathbf{v}$ associated to $P$ is positive} reduces to proving that \emph{the cut time of $\mathbf{v}$ associated to $\tilde{P}_{\mathbf{v}}$ is positive}. Indeed, suppose the distance from $\tilde{P}_{\mathbf{v}}$ to some point $x \in M \setminus \tilde{P}_{\mathbf{v}}$ is achieved along a minimizer $\gamma$ with initial velocity $\mathbf{v}$. Then, $P \subset \tilde{P}_{\mathbf{v}}$ implies that the distance from $P$ to $x$ is also achieved along the same $\gamma$.

Let us point out that for an arbitrary submanifold $P$, it may not be contained in a codimension one submanifold, as the next example demonstrates \footnote{
    Over an email correspondence, Prof. M.A. Javaloyes confirmed that they had claimed $P \subset \tilde{P}_{\mathbf{v}}$ to simplify the construction, but their proof was completely local in nature.}.

\begin{example}\label{example:sphereNoHypersurface}
    Let $M = TS^2$ be the tangent bundle of the $2$-sphere, and $N$ be the $0$-section in $M$, in particular, $N \cong S^2$. Suppose we have some $P \subset M$ of codimension $1$, with $N \subset P$. Then at each $x \in N$, we have 
    \[T_x N \subsetneq T_x P \subsetneq T_x M \Rightarrow T_x P / T_x N \subsetneq T_x M / T_x N,\]
    with each subspace having constant dimensions. In particular, we have a rank $1$ vector sub-bundle $\xi \coloneqq TP|_N / TN \subset TM|_N / N$. Since $N$ is the $0$-section of the vector bundle $M$, we have $TM|_N \cong TN \oplus M \Rightarrow TM|_N / TN \cong M$. Thus, we have $\xi$ is a vector sub-bundle of $M = TS^2$. Let us write, $TS^2 = \xi \oplus \eta$. We compute the Euler characteristic $e(TS^2) = e(\xi) \smile e(\eta) = 0$, since $e(\xi), e(\eta) \in H^1(S^2) = 0$. This is a contradiction, as $e(TS^2) = 2$. Thus, there cannot be a hypersurface of $TS^2$ containing the $0$-section.
\end{example}

Now, in the course of the proof of \cite[Theorem 3.1]{Alves2019}, the authors ended up producing $\tilde{P}_{\mathbf{v}}$ satisfying $P \cap W \subset \tilde{P}_{\mathbf{v}}$, where $W$ is some coordinates chart on $M$. Consequently, for the same $x$ and $\gamma$ as above, they obtain that the distance from $P \cap W$ to $x$ is achieved along $\gamma$, whereas we can possibly have $d(P, x) \lneq d(P \cap W, x) = d(\tilde{P}_{\mathbf{v}}, x)$ (see \autoref{fig:globalMinima}). In other words, only a \emph{local} minima of the distance function $f(p) = d(p, x)$ is achieved on $P$. Note that this is comparable to the first step in our proof of \autoref{thm:cutTimePositive}. Choosing $W$ suitably smaller, the authors were able to assert that the distance from $P$ to $x$ can actually be achieved from $P \cap W$. Indeed, this is very similar to \autoref{lemma:globalMinima} and \autoref{lemma:uniformPositive} as in the proof of \autoref{thm:cutTimePositive}.

\subsection{Relation with the Article \texorpdfstring{\cite{InnItoNagShi19}}{} by Innami et al.}\label{sec:sphereCondition}
As noted earlier, the proof of \autoref{thm:cutTimePositive} is inspired by the statement (\hyperref[eq:innamiStatement]{S}). The crux of the argument, as in \autoref{lemma:localMinima}, is that a sufficiently small sphere has large principal curvature, whereas that of the submanifold $N$ is bounded above in a compact neighborhood. Now, in order to define the principal curvature, we require at least $C^2$-smooth data. Indeed, the next example shows that the statement (\hyperref[eq:innamiStatement]{S}) can be false in low regularity.

\begin{figure}[H]
    \centering
    \def\svgwidth{0.25\columnwidth}
    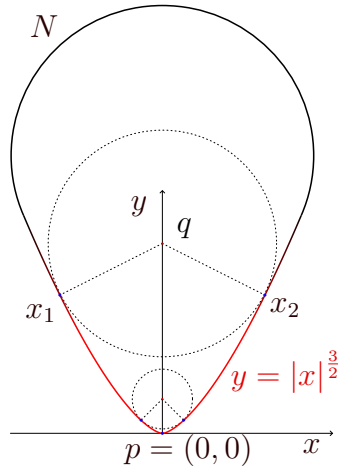

    \caption{A counterexample : we always have $q$ arbitrarily near to $p$ for which the distance to $N$ is achieved at two distinct points.}
    \label{fig:counterexample}
\end{figure}

\begin{example}\label{example:ellipseParallelCurve}
    Let us consider a closed curve in $\mathbb{R}^2$ which is $C^1$ but not $C^2$ at some point. As a concrete example, in \autoref{fig:counterexample}, we have taken the curve $y = \left\lvert x \right\rvert^{\frac{3}{2}}$ in some neighborhood of the point $p = (0, 0)$, and closed it off smoothly. Denote this curve by $N$, which is then an embedded hypersurface in $\mathbb{R}^2$. It follows that $N$ is $C^1$, but it fails to be $C^2$ at $p$. Now, for any $q \ne p$ on the $y$-axis sufficiently near $p$, there are two points on $N$, say, $x_1, x_2$ which are nearest to $q$. Thus, $q \in \mathrm{Se}(N)$, and hence, in the limit we have $p \in \mathrm{Cu}(N)$. In other words, for all such $q$, with $d(N, q)$ arbitrarily small, the sphere $S(q, d(N,q))$ intersects $N$ in two distinct points. This clearly contradicts the statement (\hyperref[eq:innamiStatement]{S})\footnote{In a personal communication, Prof. N. Innami suggested considering \emph{the parallel curve to an ellipse passing through one of its focal points}, which is also known to be $C^1$ but not $C^2$.}.
\end{example}

In this context, let us recall the related notion of interior and exterior sphere conditions. Suppose $\Omega \subset \mathbb{R}^n$ is a region, i.e., a connected and bounded open set. Recall, its topological boundary is given as $\partial \Omega = \overline{\Omega} \cap \overline{\mathbb{R}^n \setminus \Omega} = \overline{\Omega} \setminus \Omega$.
\begin{defn}\label{defn:sphereCondition}
    $\Omega$ is said to satisfy the \emph{interior} (resp. \emph{exterior}) \emph{sphere condition} at some $p \in \partial \Omega$ if there exists some $r = r(p) > 0$ and some $x$ such that the closed euclidean ball $B(x, r)$ satisfies $B(x, r) \subset \Omega$ (resp. $B(x, r) \subset \Omega^c$), and $p \in \partial B(x, r)$. If $r > 0$ can be chosen independent of $p \in \partial \Omega$, then $\Omega$ is said to satisfy the \emph{uniform} interior (resp. exterior) sphere condition.
\end{defn}

In \autoref{example:ellipseParallelCurve} we see that if $\partial \Omega$ is only $C^1$-smooth, then it might not even satisfy the sphere condition at every point. Now, it is clear that $\Omega$ satisfies the interior (resp. exterior) sphere condition at some $p \in \partial \Omega$ if for the inward (resp. outward) normal vector $\mathbf{n}$ of $\partial \Omega$ we have $\rho(\mathbf{n}) > 0$. Note that as $\Omega$ is bounded, we have $\partial \Omega$ is compact. Thus, if we assume that $\partial \Omega$ is sufficiently smooth (in fact, $C^2$-smooth suffices \cite[Lemma 14.16]{GilTru01}), then it follows from \autoref{thm:cutTimePositive} that $\Omega$ satisfies both the uniform interior and uniform exterior sphere condition. 

\begin{remark}\label{rmk:uniformSphereCondtion}
    It is known that $\Omega$ satisfies both the uniform interior and exterior condition, if and only if, $\partial \Omega$ is $C^{1,1}$ regular, i.e., locally given as the zero set of a $C^{1,1}$-regular function. A sketch of proof of this can be found in \cite[Theorem 1.0.9, pg. 7]{Barb09}. Recall that a real valued function $f$ is called $C^{1, \alpha}$-regular if it is $C^1$, and $df$ is Holder-$\alpha$ continuous. In particular, $C^{1,1}$-regularity implies $df$ is Lipschitz, and hence, differentiable almost everywhere.
\end{remark}
\ifarxiv
\appendix\section{Curvature of Small Geodesic Spheres}\label{sec:curvatureSmallBall}

In this section, we give a detailed proof of \autoref{lemma:smallBallPrincipalCurvature}. We shall require few more basic concepts from Finsler geometry.

\subsection{Reverse Finsler Metric}
Recall the \emph{reverse Finsler metric} $\bar{F}$ defined as $\bar{F}(\mathbf{v}) = F(-\mathbf{v})$ for $\mathbf{v} \in TM$. We have the following easy observation.
\begin{prop}\label{prop:reverseFinslerRelations}
    Denoting all the quantities associated to $\bar{F}$ by $\bar{g}, \bar{C}, \bar{\nabla}, \bar{R}, \bar{K}$, we have
    \[\bar{g}_{\mathbf{v}} = g_{-\mathbf{v}}, \quad \bar{C}_{\mathbf{v}} = -C_{-\mathbf{v}}, \quad \bar{\nabla}^{V} = \nabla^{-V}, \quad \bar{R}^{V} = R^{-V}, \quad \bar{K}^V = K^{-V},\]
    for any $\mathbf{v} \in \widehat{TM}$ and for any $V \in \Gamma \widehat{TM}$.
\end{prop}
\begin{proof} 
    For $\mathbf{v} \in T_p M \setminus \left\{ 0 \right\}, \mathbf{v}_1, \mathbf{v}_2 \in T_p M$ we have
    \begin{align*}
        \bar{g}_{\mathbf{v}}(\mathbf{v}_1, \mathbf{v}_2) 
        &= \left. \frac{1}{2}\frac{\partial^2}{\partial s_1 \partial s_2} \right|_{s_1 = s_2 = 0} \left( \bar{F}_p \left( \mathbf{v} + s_1 \mathbf{v}_1 + s_2 \mathbf{v}_2 \right) \right)^2 \\
        &= \left. \frac{1}{2}\frac{\partial^2}{\partial s_1 \partial s_2} \right|_{s_1 = s_2 = 0} \left( F_p \left( - \mathbf{v} - s_1 \mathbf{v}_1 - s_2 \mathbf{v}_2 \right) \right)^2 \\
        &= g_{-\mathbf{v}}\left( -\mathbf{v}_1, -\mathbf{v}_2 \right) \\
        &= g_{-\mathbf{v}}(\mathbf{v}_1, \mathbf{v}_2)
    \end{align*}
    Hence, $\bar{g}_{\mathbf{v}} = g_{-\mathbf{v}}$. Similarly, for $\mathbf{v} \in T_p M \setminus \left\{ 0 \right\}, \mathbf{v}_1, \mathbf{v}_2, \mathbf{v}_3 \in T_p M$ we have 
    \[\bar{C}_{\mathbf{v}}\left( \mathbf{v}_1, \mathbf{v}_2, \mathbf{v}_3 \right) = C_{-\mathbf{v}}\left( -\mathbf{v}_1, -\mathbf{v}_2, -\mathbf{v}_3 \right) = - C_{-\mathbf{v}}\left( \mathbf{v}_1, \mathbf{v}_2, \mathbf{v}_3 \right),\]
    and consequently, $\bar{C}_{\mathbf{v}} = -C_{-\mathbf{v}}$.
    
    Next, for arbitrary $X, Y, Z \in \Gamma TM$, we have
    \begin{align*}
        X\left( \bar{g}_V(Y, Z) \right) 
        &= X \left( g_{-V}(Y, Z) \right) \\ 
        &= g_{-V}\left( \nabla^{-V}_X Y, Z \right) + g_{-V}\left( Y, \nabla^{-V}_X Z \right) + 2 C_{-V}(\nabla^{-V}_X(-V), Y, Z) \\
        &= \bar{g}_V \left( \nabla^{-V}_X Y, Z \right) + \bar{g}_V \left( Y, \nabla^{-V}_X Z \right) - 2 C_{-V}\left( \nabla^{-V}_X V, Y, Z \right) \\
        &= \bar{g}_V \left( \nabla^{-V}_X Y, Z \right) + \bar{g}_V \left( Y, \nabla^{-V}_X Z \right) + 2 \bar{C}_V\left( \nabla^{-V}_X V, Y, Z \right) 
    \end{align*}
    Clearly,  $\nabla^{-V}_X Y - \nabla^{-V}_Y X = [Y,X]$ holds. Hence, $\bar{\nabla}^V = \nabla^{-V}$.
    
    Now, for the curvature tensor, fix some $V \in \Gamma \widehat{TM}$. We have 
    \begin{align*}
        \bar{R}^V(X, Y) Z &= \bar{\nabla}^V_X \bar{\nabla}^V_Y Z - \bar{\nabla}^V_Y \bar{\nabla}^V_X Z - \bar{\nabla}^V_{[X, Y]} Z \\
        &= \nabla^{-V}_X \nabla^{-V}_Y Z - \nabla^{-V}_X \nabla^{-V}_Y Z - \nabla^{-V}_{[X, Y]} Z \\
        &= R^{-V}(X, Y)Z,
    \end{align*}
    which proves that $\bar{R}^V = R^{-V}$.

    Lastly, for $V, W \in \Gamma \widehat{TM}$ spanning a $2$-plane field $\sigma = \mathrm{Span}\left\langle V, W \right\rangle$, we have the flag curvature
    \begin{align*}
        \bar{K}^V(\sigma)
        &= \frac{\bar{g}_V\left( \bar{R}^V(V, W)W, V \right)}{\bar{g}_V(V,V)\bar{g}_V(W,W) - \bar{g}_V(V,W)^2} \\
        &= \frac{g_{-V}\left( R^{-V}(V, W)W, V \right)}{g_{-V}(V,V)g_{-V}(W,W) - g_{-V}(V,W)^2} \\
        &= \frac{g_{-V}\left( R^{-V}(-V, W)W, -V \right)}{g_{-V}(-V,-V)g_{-V}(W,W) - g_{-V}(-V, W)^2} \\
        &=K^{-V}(\sigma).
    \end{align*}
    This concludes the proof.
\end{proof}

As a direct corollary to \autoref{prop:shapeOperator} and \autoref{prop:reverseFinslerRelations}, we now have the following.
\begin{corollary}\label{cor:shapeOperatorReversedFinsler}
    Let $\mathbf{n} \in \nu_p(N)$ for $p \in N$. If we denote by $\bar{A}_{-\mathbf{n}} : T_p N \to T_p N$ the shape operator with respect to the reversed Finsler metric $\bar{F}$ along $-\mathbf{n}$, then we have $\bar{A}_{-\mathbf{n}} = - A_{\mathbf{n}}$.
\end{corollary}
\begin{proof}
    Since $\bar{g}_{-\mathbf{n}} = g_{\mathbf{n}}$, we have $-\mathbf{n} \in \bar{\nu}_p(N)$, where $\bar{\nu}_p(N)$ is the normal cone of $N$ at $p$, with respect to $\bar{F}$. Then for any $\mathbf{x} \in T_p N$, we have $\bar{A}_{-\mathbf{n}}(\mathbf{x}) = - \left( \bar{\nabla}^{-\tilde{\mathbf{n}}}_X (-\tilde{\mathbf{n}}) \middle|_p \right)^{\top_{\bar{g}_{-\mathbf{n}}}} = \left( \nabla^{\tilde{\mathbf{n}}}_X \tilde{\mathbf{n}} \right)^{\top_{g_{\mathbf{n}}}} = -A_{\mathbf{n}}(\mathbf{x})$, which concludes the proof.
\end{proof}

\subsection{$N$-Jacobi Fields}
Using the Chern connection, we define a covariant derivative along a curve.

\begin{defn}\label{defn:covariantDerivative}\cite{Javaloyes2014}
    Given a curve $\gamma : [a,b] \rightarrow  M$ and $W \in \Gamma \gamma^*\widehat{TM}$, the \emph{covariant derivative} along $\gamma$ is defined as \[D^W_\gamma : \Gamma \gamma^* TM \rightarrow \Gamma \gamma^* TM,\]
    which satisfies the following.
    \begin{itemize}
        \item For $X, Y \in \Gamma \gamma^*TM$, and $\alpha, \beta \in \mathbb{R}$ we have 
        $D^W_\gamma(\alpha X + \beta Y) = \alpha D^W_\gamma X + \beta D^W_\gamma Y$.

        \item For $X \in \Gamma \gamma^*TM$, and $f : [a,b] \rightarrow \mathbb{R}$, we have $D^W_\gamma(f X) = \frac{df}{dt} X + f D^W_\gamma X$.
        
        \item For $X, Y \in \Gamma \gamma^* TM$, we have \[\frac{d}{dt}g_W(X,Y) = g_W \left( D^W_{\gamma} X, Y \right) + g_W \left( X, D^W_\gamma Y \right) + 2 C_W \left( D^W_\gamma W, X, Y \right).\]
    \end{itemize}
\end{defn}

A vector field $J$ along a geodesic $\gamma : [a, b] \rightarrow M$, is said to be a \emph{Jacobi field} if it satisfies the second order differential equation, called the \emph{Jacobi equation}, 
\[D^{\dot\gamma}_\gamma D^{\dot\gamma}_\gamma J - R^{\dot\gamma}(\dot\gamma, J)\dot\gamma = 0.\]
We refer to \cite{Javaloyes2014,Javaloyes2014a} for the definition of the curvature tensor $R^{\dot \gamma}(\dot\gamma, J) \dot \gamma$. Alternatively, every Jacobi field $J$ along $\gamma$ is given by a geodesic variation $\Lambda : (-\epsilon, \epsilon) \times [a,b] \rightarrow M$ of $\gamma$, via the equation
\[J(t) \coloneqq \left. \frac{\partial}{\partial s} \right|_{s=0} \Lambda(s,t), \quad t\in [a,b].\]
Recall that $\Lambda$ is a \emph{geodesic variation} of $\gamma$ if for each $s$ fixed, $\Lambda_s \coloneqq \Lambda(s, \_)$ is a geodesic, with $\Lambda_0 = \gamma$. Since the Jacobi equation is a second order ODE, given the initial data, $\mathbf{u}, \mathbf{v} \in T_{\gamma(a)}M$, there exists a unique Jacobi field $J$ along $\gamma$ satisfying, $J(a) = \mathbf{u}$ and $D^{\dot\gamma}_\gamma J(a) = \mathbf{v}$. Consequently, the collection of all Jacobi fields along $\gamma$ forms a vector space of dimension $2 \dim M$.

Given a unit-speed $N$-geodesic $\gamma : [a,b]\rightarrow M$, a vector field $J \in \Gamma\gamma^*TM$ is called an \emph{$N$-Jacobi field} if it satisfies the following initial value problem
\begin{equation}\label{eq:NJacobiEquation}
    D^{\dot\gamma}_\gamma D^{\dot\gamma}_\gamma J - R^{\gamma}(\dot\gamma, J) \dot\gamma = 0, \quad J(a)\in T_{\gamma(a)} N, \quad \dot J(a) + A_{\dot\gamma(a)} \left( J(a) \right) \in \left( T_p N \right)^{\perp_{g_{\mathbf{\dot\gamma(a)}}}}.
\end{equation}
Every $N$-Jacobi field $J$ arise from an $N$-geodesic variation $\Lambda : (-\epsilon, \epsilon) \times [a,b] \rightarrow M$ as
\begin{equation}\label{eq:jacobiVariation}
    J(t) = \left.\frac{\partial}{\partial s}\right|_{s=0} \Lambda(s,t), \quad t\in[a,b].
\end{equation}

\subsection{Estimation of the Principal Curvatures of Small Backward Spheres}
We are now in position to give a proof of \autoref{lemma:smallBallPrincipalCurvature}
\begin{proof}[Proof of \autoref{lemma:smallBallPrincipalCurvature}]
    Since $r < \mathrm{Inj}^{-}(q)$, we have $S$ is a codimension $1$ submanifold. Let $p \in S$ and $\gamma : [0, r] \to M$ be the radial geodesic joining $p = \gamma(0)$ to $q = \gamma(r)$. Let $\mathbf{n} = \dot\gamma(0)$, and note that $\mathbf{n} \in \nu_p(S)$. We have $\bar{\gamma}(t) = \gamma(r - t)$ is a geodesic with respect to the reversed Finsler metric $\bar{F}$, joining $q$ to $p$, and furthermore $-\mathbf{n} = \dot{\bar{\gamma}}(r) \in \bar{\nu}_p(S)$. Note that $r < \mathrm{Inj}^{+}(q)$ implies in particular that $p$ cannot be a conjugate point of $q$ along $\bar{\gamma}$. Let $\kappa$ be an eigenvalue of $A_{\mathbf{n}}$ with eigenvector $\mathbf{v} \in T_p S$. Then, by \autoref{cor:shapeOperatorReversedFinsler} we have, $\bar{A}_{-\mathbf{n}} \mathbf{v} = -A_{\mathbf{n}} \mathbf{v} = -\kappa \mathbf{v}$. Consider the unique Jacobi field $J$ along $\bar{\gamma}$ with $J(0) = 0$ and $J(r) = \mathbf{v}$, which exists since $p$ is not conjugate to $q$ along $\bar{\gamma}$. Since the Chern connection is torsion free, it follows from \cite{Javaloyes2014} that
    \[D^{\dot{\bar{\gamma}}}_{\bar{\gamma}}J(r) = \nabla^{\dot{\bar{\gamma}}}_{\dot{\bar{\gamma}}} J |_r = \nabla^{\dot{\bar{\gamma}}}_J \dot{\bar{\gamma}} |_r.\] 
    Hence, we have
    \[\bar{g}_{-\mathbf{n}}(\bar{A}_{-\mathbf{n}} \mathbf{v}, \mathbf{v}) = \bar{g}_{-\mathbf{n}}\left( - \left( \bar{\nabla}^{\dot{\bar{\gamma}}}_{\mathbf{v}} \dot{\bar{\gamma}} \middle|_p \right)^{\top_{\bar{g}_{-\mathbf{n}}}}, \mathbf{v} \right) = - \bar{g}_{ -\mathbf{n}}\left( \bar{\nabla}^{\dot{\bar{\gamma}}}_{\mathbf{v}} \dot{\bar{\gamma}} \middle|_p, \mathbf{v} \right) = -g_{-\mathbf{n}}\left( D^{\dot{\bar{\gamma}}}_{\bar{\gamma}} J(r), J(r) \right),\]
    and thus,
    \[-\kappa = \frac{\bar{g}_{-\mathbf{n}}(\bar{A}_{-\mathbf{n}} \mathbf{v}, \mathbf{v})}{\bar{g}_{-\mathbf{n}}(\mathbf{v}, \mathbf{v})} = - \frac{\bar{g}_{-\mathbf{n}}\left( D^{\dot{\bar{\gamma}}}_{\bar{\gamma}} J(r), J(r) \right)}{\bar{g}_{-\mathbf{n}}\left( J(r), J(r) \right)} \Rightarrow \kappa = \frac{\bar{g}_{-\mathbf{n}}\left( D^{\dot{\bar{\gamma}}}_{\bar{\gamma}} J(r), J(r) \right)}{\bar{g}_{-\mathbf{n}}\left( J(r), J(r) \right)}.\]
    Now, from our hypothesis, the flag curvature of $M$ with respect to $F$, and hence by \autoref{prop:reverseFinslerRelations}, with respect to $\bar{F}$ is bounded from above by $\lambda$. It then follows from \cite[Cor 9.8.1, pg. 254]{Bao2000} that
    \[\kappa = \frac{\bar{g}_{-\mathbf{n}}\left( D^{\dot{\bar{\gamma}}}_{\bar{\gamma}} J(r), J(r) \right)}{\bar{g}_{-\mathbf{n}}\left( J(r), J(r) \right)} \ge \mathfrak{ct}_\lambda(r),\]
    which tends to $\infty$ as $r \rightarrow 0$. Since $\kappa$ is an arbitrary eigenvalue of $A_{\mathbf{n}}$, the proof follows.
\end{proof}
\fi

\section*{Acknowledgments}
The authors would like to express their gratitude to S. Basu, I. Biswas and J. Itoh for many fruitful discussions and suggestions in the preparation of this article. They are also grateful to Prof. M. A. Javaloyes for explaining their article \cite{Alves2019} to us. The first author was supported by the NBHM grant no. 0204/1(5)/2022/R\&D-II/5649 and the second author was supported by Jilin University. 

\bibliographystyle{alphaurl}

\end{document}